\documentclass[11pt,reqno]{amsart}
\usepackage{amsmath,amsfonts,amssymb,mathabx}
\usepackage{amsthm}
\usepackage{amstext}
\usepackage[normalem]{ulem}
\usepackage{verbatim}
\usepackage{longtable}
\usepackage{tensor}
\usepackage{natbib}  
\usepackage{fancyhdr} 
\usepackage{graphicx} 
\usepackage{float}
\usepackage[a4paper,left=2cm,top=2cm,textwidth=17cm,textheight=25cm]{geometry}
\usepackage{caption}
\usepackage{bbm}

\usepackage[usenames,dvipsnames]{color}

 \definecolor{spoc}{RGB}{180, 90, 20}
\definecolor{kc}{RGB}{27, 121, 0}

\usepackage{enumitem}
\usepackage{hyperref}
\hypersetup{
	colorlinks,
	linkcolor={blue},
	citecolor={blue},
	urlcolor={blue}
}

\makeatletter
\newcommand{\setword}[2]{%
	\phantomsection
	#1\def\@currentlabel{\unexpanded{#1}}\label{#2}%
}
\makeatother

\usepackage[usenames,dvipsnames]{color}

\usepackage{subcaption}
\scrollmode

\renewcommand{\labelenumi}{(\roman{enumi})}

\newtheorem{thm}{Theorem}[section]
\newtheorem{lem}[thm]{Lemma}
\newtheorem{cor}[thm]{Corollary}
\newtheorem{rem}[thm]{Remark}
\newtheorem{rems}[thm]{Remarks}
\newtheorem{prop}[thm]{Proposition}
\newtheorem{df}[thm]{Definition}
\newtheorem{dfs}[thm]{Definitions}
\newtheorem{ex}[thm]{Example}

\newcommand{\vS}{\varSigma}
\newcommand{\vO}{\varOmega}
\newcommand{\vT}{\varTheta}
\newcommand{\uph}{\upharpoonright}
\def\N{{\mathbb N}}
\def\R{{\mathbb R}}
\def\B{{\mathfrak B}}
\def\L{{\mathcal L}}

\newcommand{\E}{\mathbb{E}}
\newcommand{\al}{\alpha}

\newcommand{\be}{\beta}
\newcommand{\vY}{\varUpsilon}
\newcommand{\1}{\mathbbm{1}}

\author{Spyridon M. Tzaninis}
\address{Department of Statistics and Insurance Science\\ University of Piraeus\\ 80 Karaoli and Dimitriou str.\\ 185 34 Piraeus\\ Greece}
\email{stzaninis@unipi.gr; bozikas@unipi.gr}
\thanks{}

\author{Apostolos Bozikas}
\thanks{}

\date{\today}

\title[Extensions of Panjer's recursion for mixed compound distributions]{Extensions of Panjer's recursion for mixed compound distributions}
\subjclass[2020]{Primary  91G05; Secondary  60E05, 28A50.} 
\keywords{Mixed compound distributions; Panjer recursion; regular conditional probabilities; exchangeability}

\begin{document}

	\begin{abstract}
In actuarial practice, the usual independence assumptions for the collective risk model are often violated, implying a growing need for considering more general models that incorporate dependence. To this purpose, the present paper studies the mixed counterpart of the classical Panjer family of claim number distributions and  their compound version, by allowing the parameters of the distributions to be viewed as random variables. Under the assumptions that the claim size process is conditionally i.i.d. and (conditionally) mutually independent of the claim counts, we provide a recursive algorithm for the computation of the probability mass function of the aggregate claim sizes. The case of a compound Panjer distribution with exchangeable claim sizes is also studied. For the sake of completeness, our results are illustrated by various numerical examples.
	\end{abstract}
	
	\maketitle

\section{Introduction}
Compound distributions are widely used in actuarial science for modelling the aggregate claims amount (denoted as $S$)  paid by an insurance company over a fixed period of time.  The computation of the probability distribution $P_S$ of $S$ is one of the main topics in the field. More formally, the random variable $S$ is defined as the random sum $S:=\sum_{n=1}^N X_n$, where $X:=\{X_n\}_{n\in\N}$ represents the individual claim sizes and $N$ is the random number of claims. Usually, the sequence $X$ is assumed to be independent and identically distributed and independent of $N$.  Even under these independence assumptions, the determination of $P_S$   can be a difficult task, since one usually has to compute higher order convolutions. Consequently, recursive algorithms are often used for the computation of $P_S$. \smallskip

One of the most prominent algorithms in actuarial literature is Panjer's recursion (see \cite{pa}, Theorem, p. 24), which in the case of discrete claims size distribution simplifies the computation of the probability mass function (abbreviated as pmf) of $S$. Recall that the Panjer class of order $k\in\N_0$, denoted by ${\bf Panjer}(a,b;k) $, is a collection of claim number distributions $P_N=\{p_n\}_{n\in\N_0}$, satisfying $p_n=0$ for any $n\leq k-1$  and 
\begin{equation}
	p_n=\left(a+\frac{b}{n}\right)\cdot p_{n-1}\quad\text{ for all } n\geq k+1
	\label{orpa}
\end{equation}
(see \cite{hls}, p. 283).	The members of the aforementioned class have been characterized by  \cite{sj} for $k=0$,  \cite{wil88} for $ k=1 $ and \cite{hls} for a general $k \geq 2$. Recently, \cite{fac} proposed a reparametrization  for the members of the ${\bf Panjer}(a,b;k)$ class for $k\geq 0$ resulting in a more practical and unified representation of the corresponding probability distributions. More general recursive expressions than \eqref{orpa} have also been studied by  \cite{hessa}, \cite{sch}, \cite{su} and \cite{ws}. For an extensive review on recursive formulas for actuarial applications, we refer to \cite{sv} and the references therein.\smallskip

In the case of inhomogeneous insurance portfolios, mixtures of claim number distributions are considered for modelling the claim counts. \cite{gy}, \cite{gdsco}, \cite{matal}, \cite{tzougas} and \cite{zwang} proposed some (compound) mixed negative binomial distributions with different parametrizations using various mixing distributions. On the other hand, the literature for mixed Poisson and compound mixed Poisson distributions is much more enriched, encompassing,  among others, the works of \cite{anch}, \cite{kaxe}, \cite{nakp1, nakp2}, \cite{sv2004}, \cite{tzougas2} and \cite{wil1986, wil1987}. In particular, recursive formulas have been obtained by \cite{wil} under the additional assumption that the logarithm of the induced (by the mixing distribution) probability density function can be written as the ratio of two polynomials. Moreover, by adopting the usual independence assumptions, \cite{hessb} obtained a recursive formula for the computation of $P_S$. 
\smallskip 

In actuarial practice, the assumptions of independence among the random variables of $X$ and the mutual independence between $X$ and $N$ are often violated (see among others \cite{albo}, \cite{alte} and \cite{kopa}). In particular, the mutual independence assumption seems to be unrealistic, especially when considering inhomogeneous portfolios. This motivates us to consider a dependence structure among the random variables of $ X $ and between $ N $ and $ X $. In contrast to the existing notion of compound (mixed) counting distributions, where the sequence $X$ is considered to be i.i.d. and mutually independent of $N$ (see \cite{gr}, Chapter 8, and \cite{lm3ar}, p. 4), in our case $X$ is only conditionally i.i.d. and conditionally mutually independent of $N$. This leads to the concept of \textit{mixed compound distributions} (see Section \ref{cd}), which constitute a strong generalization of compound (mixed) distributions. \smallskip

Under the above framework, this work studies the mixed counterpart of the original ${\bf Panjer}(a,b;0)$ class and the corresponding compound distributions, by allowing  the parameters of the claim number distribution $P_N$ to be measurable functions of a random vector, the claim size process $X$ to be  conditionally i.i.d. and conditionally mutually independent of $ N $. The above assumptions allow for a (possible) correlation (positive or negative) between $ N $ and $ X $,  as well as  correlation among the random variables of $X$. In addition, we consider the case of compound Panjer distributions with exchangeable claims.  Henceforth, we denote by ${\bf Panjer}(a(\vT),b(\vT);0)$ the mixed Panjer class, where $\vT$ is a $ d $-dimensional ($ d \in \N $) random vector (also referred as structural parameter). \smallskip

The rest of this paper is organized as follows. Section \ref{mcnd} provides a characterization for the members of  ${\bf Panjer}(a(\vT),b(\vT);0)$ class in terms of regular conditional probabilities, and a recursive algorithm for the computation of their pmf
Section \ref{cd} studies the probabilistic aspects of mixed compound distributions. Section \ref{rcd} presents a recursive algorithm for the computation of $P_S$ in the case that $P_N\in{\bf Panjer}(a(\vT),b(\vT);0)$ 
and  the claim size distribution is concentrated on $\N_0$. Numerical applications are also given in Section \ref{rcd}, whereas Section \ref{conc} concludes the paper.

\section{The class of mixed Panjer claim number distributions}\label{mcnd}

$\N$ and $\R$ stand for the natural and the real numbers, respectively, while  $\N_0:=\N\cup\{0\}$ and  $\R_+:=\{x\in\R: x\geq{0}\}$. If $d\in\N$, then $\R^d$ denotes the Euclidean space of dimension $d$. Given a topology $\mathfrak{T}$  on a set $\vY$, write $\B(\vY)$ for its Borel $\sigma$-algebra on $\vY$, i.e., the $\sigma$-algebra generated by $\mathfrak{T}$. Also $\B:=\B(\R)$ and $\B(\alpha,\beta):=\B\big((\alpha,\beta)\big)$, where $\alpha,\beta\in\R$ with $\alpha<\beta$, are the Borel $\sigma$-algebras of subsets of $\R$ and $(\alpha,\beta)$, respectively. For a non-empty set $A$, denote by $\1_A$ and by $\mathcal P(A)$ its indicator function and its power set, respectively. Also, for a map $f:A\rightarrow{E}$ and a non-empty set $B\subseteq{A}$, write $f{\uph} B$ to denote the restriction of $f$ to $B$. Given two measurable spaces $(\vO,\vS)$ and  $(\vY,T)$, as well as a $\vS$-$T$-measurable map $Z:\vO\to\vY$, denote by $\sigma(Z):=\{Z^{-1}[B] : B\in T\}$ the $\sigma$-algebra generated by $Z$. 
\smallskip 

\textit{Throughout what follows  $(\vO,\vS,P)$ is an arbitrary probability space and  $N$ is a  random variable on $\vO$  taking values in a subset $R_N$ of \ $\N_0$  containing  $0$.}\smallskip

\cite{sj}  proved  that the only non-degenerate members of the original Panjer class of claim number distributions ${\bf Panjer}(a,b;0)$ are the  Poisson, the binomial and the negative binomial distributions (see \cite{sj}, Theorem 1). These distributions will be referred as {\bf basic claim number distributions}. Since the members of ${\bf Panjer}(a,b;0)$ satisfy the recursive condition \eqref{orpa}, an interesting question  is whether their mixtures also satisfy a recursive relation. The previous question  traces back to \cite{su}, but for a more general family of claim number distributions, where it seems not to have a definite answer. \smallskip

A way to model the mixtures of claim number distributions is to assume the existence of a random variable (or more generally of a random vector)   $\vT$  on $\vO$ with values in $D\subseteq [0,\infty)$, such that the conditional distribution of $N$ given $\vT$ satisfies condition $P_{N\mid\vT}={\bf K}(\vT)$ $P{\uph}\sigma(\vT)$-almost surely (written a.s. for short), where ${\bf K}(\vT)$ is a conditional  distribution (cf., e.g., \cite{lmt1}, p. 455, for the definition of a conditional distribution). A claim number distribution $P_N$ that satisfies the previous condition  is said to be a {\bf mixed claim number distribution}.  This motivates us to provide the following definitions within the class of mixed claim number distributions.   

\begin{dfs}\normalfont
\label{micnd}
Let $\vT$ be a $d$-dimensional random vector with values in $D\subseteq \R^d$. A random variable $N$ is distributed according to: \smallskip
	
\noindent \textbf{(a)}  the {\bf mixed Poisson distribution} with structural parameter $\xi(\vT)$ (written {\bf MP}$\big(\xi(\vT)\big)$ for short), where $\xi$ is a $\B(D)$-$\B(\R_+)$-measurable function,  if  
\[
p_n(\vT):=P(N=n\mid\vT)=e^{-\xi(\vT)}\cdot\frac{\big(\xi(\vT)\big)^n}{n!}\quad P{\uph}\sigma(\vT)\text{-a.s.}
\] 
for any $n\in\N_0$; \smallskip
	
\noindent \textbf{(b)} the \textbf{mixed binomial}  distribution  with structural parameters $z_1(\vT)$ and $z_2(\vT)$ (denoted by {\bf MB}$\big(z_1(\vT),z_2(\vT)\big)$ for short), where $z_1$ and  $z_2$ are $\B(D)$-$\mathcal P(\N_0)$- and  $\B(D)$-$\B(0,1)$-measurable functions, respectively, if 
\[
p_n(\vT)=\binom{z_1(\vT)}{n}\cdot \big(z_2(\vT)\big)^n\cdot \big(1-z_2(\vT)\big)^{z_1(\vT)-n} \quad P{\uph}\sigma(\vT)\text{-a.s.}
\]
for any $n\in\{0,1,\ldots,z_1(\vT)\}$ $P{\uph}\sigma(\vT)$-a.s.. In particular, if $z_1(\vT)$ or $z_2(\vT)$  is degenerate, simply write {\bf MB}$\big(m,z_2(\vT)\big)$ ($m\in\N$) or {\bf MB}$\big(z_1(\vT),p\big)$ ($p\in(0,1)$), respectively; \smallskip
	
\noindent \textbf{(c)}  the \textbf{mixed negative binomial}  distribution  with structural parameters $\rho_1(\vT)$ and $\rho_2(\vT)$ (written {\bf MNB}$\big(\rho_1(\vT),\rho_2(\vT)\big)$ for short), where $\rho_1$ and  $\rho_2$ are $\B(D)$-$\B(\R_+)$- and  $\B(D)$-$\B(0,1)$-measurable functions, respectively, if  
\[
p_n(\vT)=\frac{\Gamma\big(\rho_1(\vT)+n\big)}{n!\cdot\Gamma\big(\rho_1(\vT)\big)} \cdot  \big(\rho_2(\vT)\big)^{\rho_1(\vT)}\cdot \big(1-\rho_2(\vT)\big)^{n}   \quad P{\uph}\sigma(\vT)\text{-a.s.}
\]
for any $n\in\N_0$. In particular, if $\rho_1(\vT)$ or $\rho_2(\vT)$ is degenerate, simply write {\bf MNB}$\big(r,\rho_2(\vT)\big)$ ($r>0$) or {\bf MNB}$\big(\rho_1(\vT),p\big)$ ($p\in(0,1)$), respectively.
\end{dfs} 

\textit{In what follows, unless stated otherwise, $\vT$ is a $d$-dimensional random vector with values in $D\subseteq \R^d$ and $\xi, z_1,z_2,\rho_1$ and $\rho_2$ are as in Definitions \ref{micnd}.}

\begin{df} \label{mp} 
\normalfont
A claim number distribution $P_N$ belongs to the class ${\bf Panjer}(a(\vT),b(\vT);0)$ if, the conditional distribution of $N$  given  $\vT$ satisfies for every $n\in R_N\setminus\{0\}$ the condition 
\begin{gather}
p_n(\vT)=\bigg(a(\vT)+\frac{b(\vT)}{n}\bigg)\cdot p_{n-1}(\vT)\quad P{\uph}\sigma(\vT)\text{-a.s.},
\label{mpa1}
\end{gather}
where  $a$ and $b$ are real-valued $\B(D)$-measurable functions. 
\end{df}

\begin{rems}
	\label{mpar}
	\normalfont
	{\bf (a)} In the special case $P_\vT=\delta_{\theta_0}$, where  $\delta_{\theta_0}$ denotes the Dirac measure concentrated on $\theta_0\in D$, the class ${\bf Panjer}(a(\theta_0),b(\theta_0);0)$  coincides with the original Panjer class of claim number distributions. \smallskip
	
	\noindent {\bf (b)} Recursive formulas for mixed claim number distributions can be found also in \cite{gsw}, Lemma 5.7, but in a different context from condition \eqref{mpa1}. More precisely, these recursive formulas arise from a change of distributions technique for the mixing distribution. 
\end{rems}

Since the definition of ${\bf Panjer}(a(\vT),b(\vT);0)$ involves conditioning, it is natural to expect that the notion of regular conditional probabilities (also known as disintegrations) will play a crucial role in order to avoid trivialities and to treat conditioning in a rigorous way.  The following definition is a special instance of \cite{fr4}, Definition 452E, appropriately adapted for our purposes (see also \cite{lm3}, Definition 3.1).

\begin{df}\label{rcp} 
\normalfont
Let $(\vY,T,Q)$ be a probability space. A family $\{P_y\}_{y\in \vY}$ of probability measures on $\vS$ is called a {\bf regular conditional probability} (written rcp for short) of $P$ over $Q$  if
\begin{itemize}
\item[{\bf(d1)}] for each $E\in\vS$ the map $y\mapsto P_y(E)$ is $T$-measurable;
\item[{\bf(d2)}] $\int P_{y}(E)\,Q(dy)=P(E)$ for each $E\in\vS$.
\end{itemize}
	
If $f:\vO\to\vY$ is an inverse-measure-preserving function (i.e., $P\big(f^{-1}[B]\big)=Q(B)$ for each $B\in T$), an rcp $\{P_y\}_{y\in \vY}$ of $P$ over $Q$ is called \textbf{consistent} with $f$ if, for each $B\in T$, the equality $P_y\big(f^{-1}[B]\big)=1$ holds for $Q$-almost all (written a.a. for short) $y\in B$. 
\end{df}

The following result is taken from  \cite{lm1v}, and serves as a basic tool for the proofs of the upcoming results.  In order to present it, denote by $\circ$  the function composition operator.

\begin{lem}
	\label{35}
	{\normalfont(\cite{lm1v},  Lemma 3.5)} Let $(\vY,T,Q)$ be a probability space and  $f:\vO\to\vY$ be an inverse-measure-preserving function. Put $\sigma(f):=\{f^{-1}[B] : B\in T\}$ and suppose that an rcp $\{P_y\}_{y\in \vY}$  of $P$ over $Q$  consistent with $f$ exists. Then, for each $A\in\vS$ and $B\in T$ the following hold true:
	\begin{enumerate}
		\item $\E_P[g\mid\sigma(f)] =\E_{P_\bullet}[g]\circ f $ $P{\uph}\sigma(f)$-a.s., where $g$ is a $\vS$-$T$-measurable function such that $\int g\,dP$ is defined on $\R\cup\{-\infty,\infty\}$;
		\item$\int_B P_y(A)\,Q(dy)=\int_{f^{-1}[B]}\E_P[\1_A\mid\sigma(f)]\,dP=P\big(A\cap f^{-1}[B]\big)$.
	\end{enumerate} 
\end{lem}

{\em Henceforth, $\{P_\theta\}_{\theta\in D}$ is an  rcp of $P$ over $P_\vT$ consistent with $\vT$.  For any $n\in R_N$ and $\theta\in D$, set $p_n(\theta):=P_\theta(N=n)$.}

\begin{rems}
	\label{rcprem}
	\normalfont
	\textbf{(a)} The use of rcp's in applied probability has been criticized mainly due to the fact that their existence is not always guaranteed (cf., e.g., \cite{sto}, Subsection 2.4). However,  if $\vS$ is countably generated and $P$ is perfect (see \cite{fr4}, 451A(d), for the definition of a perfect measure), then there always exists an rcp $\{P_y\}_{y\in\vY}$ of $P$ over $Q$ consistent with every inverse-measure-preserving function $f$ from $\vO$ into $\vY$, provided that $T$ is also countably generated (see \cite{faden}, Theorems 6 and 3). Note that the most important applications in actuarial science are still rooted in the case that $\vO$ is a Polish space (cf., e.g., \cite{Co}, p. 239, for its definition), where such rcp's always exist. In particular, $\R^d$ and $\R^{\N}$ are typical examples of Polish spaces appearing in applications (see also \cite{lm3}, Remark 3.2). For an excellent review on rcp's and their applications, we refer the interested reader to \cite{chpo}.\smallskip
	
\noindent {\bf (b)} Each of the probability measures $P_\theta$ of the corresponding rcp can be interpreted as a version of the conditional probability $P(\bullet\mid\vT=\theta)$,  which is well-defined, up to an a.s. equivalence, as a function of $\theta$, while the consistency of $\{P_\theta\}_{\theta\in D}$ with $\vT$ can be interpreted as the concentrating property $P(\vT\neq\theta\mid \vT=\theta)=0$ for any $\theta\in D$ of conditional probabilities.\smallskip
	
\noindent \textbf{(c)} Assume that $\vT$ is a positive random variable and let $P_N$ be a mixed claim number distribution. If the map $\theta\mapsto\E_{P_\theta}[N]$ is increasing (resp. decreasing) for $P_\vT$-a.a. $\theta>0$, then the random variables $N$ and $\vT$ are $P$-positively (resp. negatively) correlated.\smallskip
	
	In fact, assume first that the map $\theta\mapsto\E_{P_\theta}[N]$ is increasing for $P_\vT$-a.a. $\theta>0$.  By applying Lemma \ref{35}, we get that $\E_P[N\cdot\vT]=\E_{P_\vT}\big[\theta\cdot\E_{P_\theta}[N]\big]$, implying together with \cite{Sc2014}, Theorem 2.2, and Lemma \ref{35}, that  $\E_P[N\cdot\vT]\geq\E_P[N]\cdot \E_P[\vT]$; hence $N$ and $\vT$ are positively correlated. The proof for a decreasing map $\theta\mapsto\E_{P_\theta}[N]$ is similar.
\end{rems}

It is natural to ask which are the members of {\bf Panjer}$(a(\vT),b(\vT);0)$, and if  they can be characterized in terms of the basic claim number distributions. By applying Lemma \ref{35}, the next result yields a characterization for the members of the ${\bf Panjer}(a(\vT),b(\vT);0)$ class via rcp's, connecting in this way the defining property of the original Panjer class and of its mixed counterpart. This characterization also serves as a useful tool for proving the results presented throughout the paper. 
\begin{prop}
\label{mpanchar}
The following statements are equivalent:
\begin{enumerate}
\item $P_N\in {\bf Panjer}(a(\vT),b(\vT);0)$;
\smallskip
\item $(P_\theta)_N\in {\bf Panjer}(a(\theta),b(\theta);0)$ for $P_\vT$-a.a. $\theta\in D$. 
\end{enumerate}
\end{prop}

\begin{proof}
Fix on arbitrary $n\in R_N\setminus\{0\}$ and $F\in\B(D)$. Statement (i) is equivalent to 
\[
\E_P\big[\1_{\vT^{-1}[F]}\cdot p_n(\vT)\big]\, dP=\E_{P}\bigg[\1_{\vT^{-1}[F]}\cdot \bigg(a(\vT)+\frac{b(\vT)}{n}\bigg)\cdot p_{n-1}(\vT)\bigg]
\] 
or, according to Lemma \ref{35}(i),  to 
\[
\E_{P_\vT}\big[\1_F\cdot p_n(\theta)\big]=\E_{P_\vT}\bigg[\1_{F}\cdot \bigg(a(\theta)+\frac{b(\theta)}{n}\bigg)\cdot p_{n-1}(\theta)\bigg];
\] 
hence we equivalently get  condition 
\[
p_{n}(\theta)= \bigg(a(\theta)+\frac{b(\theta)}{n}\bigg)\cdot p_{n-1}(\theta) \quad\text{for } P_\vT\text{-a.a. } \theta\in D,
\] 
which is equivalent to statement (ii). This completes the  proof.
\end{proof}

\begin{rem}
\label{elem}
\normalfont
In other words, Proposition \ref{mpanchar} demonstrates that a claim number distribution $P_N$ is an element of ${\bf Panjer}(a(\vT),b(\vT);0)$ if and only if it is a mixture of a $P_\theta$-basic claim number distribution. The next table summarizes the members of the mixed Panjer family.
\begin{table}[H]
\centering
\caption{The members of ${\bf Panjer}(a(\vT),b(\vT);0)$ and their corresponding functions $a$ and $b$.}
\def\arraystretch{1.5}
\begin{tabular}{| c | c | c | }
\hline
$P_N$ & $a(\vT)$ & $b(\vT)$    \\   \hline  
${\bf MP}\big(\xi(\vT)\big)$  & 0 & $\xi(\vT)$    \\ \hline 
${\bf MB}\big(z_1(\vT),z_2(\vT)\big)$ & $-\frac{z_2(\vT)}{1-z_2(\vT)}$ & $\big(z_1(\vT)+1\big)\cdot \frac{z_2(\vT)}{1-z_2(\vT)}$  \\ \hline
${\bf MNB}\big(\rho_1(\vT),\rho_2(\vT)\big)$ & $1-\rho_2(\vT)$ & $\big(\rho_1(\vT)-1\big)\cdot\big(1-\rho_2(\vT)\big)$  \\
\hline
\end{tabular}
\label{pt}
\end{table}
 	
\end{rem}

Recall that the so-called Neyman Type A distribution arises as a  mixed Poisson  distribution with Poisson distributed structure parameter (cf., e.g., \cite{jkk}, Section 9.6 for more details). The next example presents a recursive formula for such a mixture of distributions (see also \cite{beal}, condition (4)) by applying the characterization appearing in Proposition \ref{mpanchar}.
\begin{ex}
	\label{MP27}
	\normalfont
	Let $D=\N_0$ and take $P_{N}={\bf MP}(\vT)$  with $P_\vT={\bf P}(a)$ ($a>0$). Fix an arbitrary $n\in\N$. Applying Proposition \ref{mpanchar}, along with Lemma \ref{35}, yields
	\[
	p_n=\frac{a\cdot e^{-1}}{n}\cdot\sum_{k=0}^{n-1}\frac{1}{k!}\cdot\bigg(\sum_{\theta=0}^\infty  e^{-\theta}\cdot\frac{\theta^{n-1-k}}{(n-1-k)!}\cdot e^{-a}\cdot \frac{a^{\theta}}{\theta!}\bigg)
	\]
	which implies
	\[
	p_n=\frac{a\cdot e^{-1}}{n}\cdot\sum_{k=0}^{n-1}\frac{1}{k!}\cdot p_{n-1-k}.
	\]
	The initial value for the recursion is given by the formula
	\[
	p_0=\E_P\big[e^{-\vT}\big]=e^{-a\cdot(1-e^{-1})}.
	\]
\end{ex}

\begin{rem}
	\label{mpoisson}
	\normalfont 
	In the case that $N$ is \textbf{mixed Poisson distributed with structure distribution $U$}, i.e., 
	\[
	p_n=\int_0^\infty e^{-\theta}\cdot \frac{\theta^n}{n!}\, U(d\theta)\quad\text{ for any }n\in\N_0
	\]
(cf., e.g., \cite{lmt1}, Definitions 3.1(b)), \cite{wil} and \cite{hessb} have provided some important results concerning recursive formulas, under the additional assumption that the logarithm of the mixing density can be written as the ratio of two polynomials.  However, these  results cannot, in general, be transferred to the class of  \textbf{MP}$(\vT)$, as it is not always possible given an arbitrary probability space $(\vO,\vS,P)$ and a probability distribution $U$ to construct a random variable $\vT$ on $\vO$, such that $P_\vT=U$ (see \cite{lm3},  Remark 4.5). Even if one assumes the  existence of $\vT$, it is not, in general, possible to construct an rcp of $P$ over $U$ consistent with $\vT$ (see \cite{lmt1}, Examples 4 and 5).  Nevertheless, for the most interesting cases appearing in applications,  there exist a random variable $\vT$ such that $P_\vT=U$ and an rcp of $P$ over $U$ consistent with $\vT$ (see \cite{lm5}, Theorem 3.1, and \cite{lmt1}, Examples 1, 2 and 3). 
\end{rem}

It is evident that if $P_N\in{\bf Panjer}(a(\vT),b(\vT);0)$, then by applying Proposition \ref{mpanchar} we get 
\[
p_n=\int_D \bigg(a(\theta)+\frac{b(\theta)}{n}\bigg)\cdot p_{n-1}(\theta)\,P_\vT(d\theta)\quad\text{for any } n\in R_N{\setminus}\{0\}.
\]   
For various choices of the functions $a$ and $b$ (see Table \ref{pt}), we can obtain recursive  formulas for the computation of the probabilities $p_n$, see e.g., \cite{gdsco}, \cite{wil} and \cite{zwang}.  However a unified formula for these recursions does not exist in the literature; hence a question that naturally arises is whether we can obtain a recursion of the form 
\[
p_n=C_n\cdot p_{n-1} \quad\text{for any } n\in R_N{\setminus}\{0\},
\]	
where $C_n\in(0,\infty)$, for the members of ${\bf Panjer}(a(\vT),b(\vT);0)$. The next theorem introduces an appropriate rcp of $P$ over $P_N$ consistent with $N$ in order to obtain such a recursive formula.
\begin{thm}
	\label{rec}
	Let $P_N\in {\bf Panjer}(a(\vT),b(\vT);0)$. There exists a family  $\{\mu_n\}_{n\in{R}_N}$ of probability measures on $\vS$ defined by means of
	\[
	\mu_n(A):=\frac{\E_P[\1_A\cdot\1_{\{N=n\}}]}{\E_P[\1_{\{N=n\}}]}\quad\text{for any } A\in\vS
	\] 
	being an rcp of $P$ over $P_N$ consistent with $N$, such that for any $n\in{R}_N\setminus\{0\}$ and $F\in\B(D)$,  the condition
 \[
p_n=C_n\cdot p_{n-1}
 \]
holds true, where 
 \[
 C_n:=\E_{\mu_{n-1}}\bigg[a(\vT)+\frac{b(\vT)}{n}\bigg].
 \]
\end{thm}

\begin{proof}
	Fix on an arbitrary $n\in R_N$ and define the set-function $\mu_n:\vS\to(0,\infty)$ by	
	\[
	\mu_n(A):=\frac{\E_P[\1_A\cdot\1_{\{N=n\}}]}{\E_P[\1_{\{N=n\}}]}\quad\text{for any } A\in\vS.
	\] 	
	A standard computation justifies that $\mu_n$ is a probability measure on $\vS$. Let $A\in\vS$ be fixed, but arbitrary. First note that condition (d1) holds since the map $n\mapsto \mu_n(A)$ is  clearly $\mathcal P(R_N)$-measurable while condition (d2) follows from 
	\[
	P(A)=\E_P\big[\E_P[\1_A\mid N]\big]=\int \frac{\E_P[\1_A\cdot\1_{\{N=n\}}]}{\E_P[\1_{\{N=n\}}]}\,dP_N=\int \mu_n(A)\,dP_N;
	\]
	hence the family $\{\mu_n\}_{n\in{R}_N}$ of probability measures is an rcp of $P$ over $P_N$. The consistency of $\{\mu_n\}_{n\in R_N}$ follows immediately by the definition of $\mu_n$. \smallskip
	
As $P_N\in {\bf Panjer}(a(\vT),b(\vT);0)$ and $\{P_\theta\}_{\theta\in D}$ is an rcp of $P$ over $P_\vT$ consistent with $\vT$, we can apply Proposition \ref{mpanchar}, together with Lemma \ref{35}, to get that 
	\begin{gather}
		p_{n+1}=\E_{P_\vT}\bigg[\bigg(a(\theta)+\frac{b(\theta)}{n+1}\bigg)\cdot  p_{n}(\theta)\bigg].
		\label{2121}
	\end{gather}
But	since  
	\begin{align*}
		\E_{\mu_n}[\1_{\vT^{-1}[F]}]=\frac{\E_P[\1_{\vT^{-1}[F]}\cdot\1_{\{N=n\}}]}{\E_P[\1_{\{N=n\}}]}=\frac{\E_{P_\vT}[\1_F\cdot p_{n}(\theta)]}{p_n}\quad\text{for every } F\in\B(D),
	\end{align*} 
where the second equality follows by Lemma \ref{35}, we obtain 
	\[
	\E_{P_\vT}\bigg[\bigg(a(\theta)+\frac{b(\theta)}{n+1}\bigg)\cdot  p_{n}(\theta)\bigg]=\E_{\mu_{n}}\bigg[\bigg(a(\vT)+\frac{b(\vT)}{n+1}\bigg)\bigg]\cdot p_n.
	\]
	The latter  together with condition \eqref{2121} completes the proof. 
\end{proof}

In the next set of examples, we apply Theorem \ref{rec} in order to obtain recursive formulas for the pmf of a mixed counting distribution  that belongs to  ${\bf Panjer}(a(\vT),b(\vT);0)$.  In our first example, we revisit the case of  the  mixed Poisson distribution.  

\begin{ex}\label{MP213}
\normalfont
Fix an arbitrary $n\in\N$ and let $P_{N}={\bf MP}\big(\xi(\vT)\big)$. Since $P_N$ is an element of ${\bf Panjer}(a(\vT),b(\vT);0)$ with $a(\vT)=0$ and $b(\vT)=\xi(\vT)$ (see Table \ref{pt}), we apply Theorem \ref{rec} to get 
\[ 
p_n=\E_{\mu_{n-1}}\bigg[\frac{\xi(\vT)}{n}\bigg]\cdot p_{n-1}=\frac{\E_P\big[\big(\xi(\vT)\big)^n\cdot e^{-\xi(\vT)}\big]}{n\cdot\E_P\big[\big(\xi(\vT)\big)^{n-1}\cdot e^{-\xi(\vT)}\big]}\cdot p_{n-1}.
\]
Setting $\widehat u_{\xi}(t):=\E_P\big[e^{-t\cdot\xi(\vT)}\big]$ ($t\geq 0$),  the latter equality becomes
\begin{gather}
p_n = - \frac{1}{n}\cdot\frac{\widehat u_{\xi}^{(n)}(1)}{\widehat u_{\xi}^{(n-1)}(1)}\cdot  p_{n-1}, 
\label{mp2131} 
\end{gather}
where notation $f^{(n)}$ stands for the $n$-th order derivative of a real-valued function $f$. Note that the initial value for the recursion is 
\[
p_0=\widehat u_{\xi}(1).
\] 
	
In particular, let $D=(0,\infty)$, $\xi=\text{id}_D$, where by $\text{id}_D$ is denoted the identity map on $D$, and consider the gamma distribution with parameters  $\al,\be>0$ (written \textbf{Ga}$(\al,\be)$ for short)  defined as
\[
\textbf{Ga}(\al,\be)(B)=\int_B \frac{\be^\al}{\Gamma(\al)}\cdot x^{\al-1}\cdot e^{-\be\cdot x}\,\lambda(dx)\quad\text{for any } B\in\B(0,\infty),
\]
where $\lambda$ denotes the restriction to $\B(0,\infty)$ of the Lebesgue measure on $\B$. Take  $P_\vT=w_1\cdot{\bf Exp}(\be)+w_2\cdot{\bf Ga}(\al,\be)$ ($\al,\be>0$), where $w_1,w_2\in [0,1]$ with $w_1+w_2=1$.   Since
	\[
	\widehat{u}(t):=\widehat{u}_{\text{id}_D}(t)=w_1\cdot\frac{\be}{\be+t}+w_2\cdot\Big(\frac{\be}{\be+t}\Big)^\al\quad\text{for any } t\geq 0
	\]
	it follows easily that 
	\[
	\widehat{u}^{(n)}(t)=\frac{(-1)^n}{(\be+t)^{\al+n}}\cdot\Big(w_1\cdot n!\cdot\be\cdot(\be+t)^{\al-1}+w_2\cdot(\al)_n\cdot\be^\al \Big)\quad\text{for any } t\geq 0,
	\]
	where $(x)_k$, $x\in\R$ and $k\in\N$ denotes the ascending factorial (cf., e.g., \cite{jkk}, p. 2); hence condition \eqref{mp2131} yields the recursive formula
	\begin{gather}
		p_n=\frac{1}{n\cdot(\be+1)}\cdot\frac{w_1\cdot n!\cdot(\be+1)^{\al-1}+w_2\cdot(\al)_n\cdot\be^{\al-1}}{w_1\cdot (n-1)!\cdot(\be+1)^{\al-1}+w_2\cdot(\al)_{n-1}\cdot\be^{\al-1} } \cdot p_{n-1} 
		\label{lin1}
	\end{gather}
	with initial value
	\begin{gather}
		p_0=w_1\cdot\frac{\be}{\be+1}+w_2\cdot\Big(\frac{\be}{\be+1}\Big)^\al.
		\label{lin0}
	\end{gather}
	
	\noindent In the special case   $w_1=0$ and $w_2=1$, conditions \eqref{lin1} and \eqref{lin0} yield 
	\[
	p_n=\frac{\al+n-1}{n\cdot(\be+1)} \cdot  p_{n-1}\quad\text{with}\quad p_0=\Big(\frac{\be}{\be+1}\Big)^\al,
	\]	
	which is the usual  negative  binomial  recursion (see \cite{pa}, p. 23).\smallskip
	
	Now take $w_1=\frac{\be}{\be+1}$, $w_2=\frac{1}{\be+1}$ and $\al=2$. For this particular case we get 
	\[
	P_\vT(B)=\int_B \frac{\be^2}{\be+1}\cdot(\theta+1)\cdot e^{-\be\cdot\theta}\,\lambda(d\theta)\quad\text{for any } B\in\B(0,\infty)
	\] 
	implying that the structural parameter is distributed according to  the Lindley distribution (cf., e.g., \cite{gan} for the definition and its basic properties).   Conditions \eqref{lin1} and \eqref{lin0} give
	\[
	p_n=\frac{\be+n+2}{(\be+1)\cdot (\be+1+n)} \cdot p_{n-1}\quad\text{with}\quad p_0=\frac{\be^2\cdot(\be+2)}{(\be+1)^3}.
	\]	
\end{ex}

The next example presents a general recursive formula for the ${\bf MB}\big(m,z_2(\vT)\big)$ ($m\in\N$) distribution. As a special case, we obtain a well-known recursive formula for the negative hypergeometric distribution (cf., e.g., \cite{jkk}, Section 6.2.2, for the definition  and its properties) appearing in \cite{hessa}, Example 3.

\begin{ex}
\label{MB213}
\normalfont
Letting  $P_{N}={\bf MB}\big(m,z_2(\vT)\big)$ ($m\in\N$), we fix on an arbitrary $n\in\{1,\ldots, m\}$. Since $P_N\in {\bf Panjer}(a(\vT),b(\vT);0)$ with $a(\vT)=-\frac{z_2(\vT)}{1-z_2(\vT)}$ and $b(\vT)=(m+1)\cdot\frac{z_2(\vT)}{1-z_2(\vT)}$ (see Table \ref{pt}), we can apply Theorem \ref{rec} to get 
\begin{align}
p_n&=\frac{m-n+1}{n}\cdot\E_{\mu_{n-1}}\bigg[\frac{z_2(\vT)}{1-z_2(\vT)}\bigg]\cdot p_{n-1}\nonumber\\
&=\frac{m-n+1}{n}\cdot\frac{\E_P\big[\big(z_2(\vT)\big)^n\cdot\big(1-z_2(\vT)\big)^{m-n}\big]}{\E_P\big[\big(z_2(\vT)\big)^{n-1}\cdot\big(1-z_2(\vT)\big)^{m-n+1}\big]}\cdot  p_{n-1}.
\label{mb2131} 
\end{align}
In this case, the initial value is given  by
\[
p_0=\E_P\big[\big(1-z_2(\vT)\big)^m\big].
\]
	
In particular, let $D=(0,1)$, let $z=\text{id}_D$ and take $P_\vT={\bf Beta}(\alpha,\beta)$ for $\alpha,\beta>0$ (cf., e.g., \cite{Sc}, p. 179, for the definition of the beta distribution). An easy computation yields 
\[
\frac{\E_P[\vT^n\cdot(1-\vT)^{m-n}]}{\E_P[\vT^{n-1}\cdot(1-\vT)^{m-n+1}]}=\frac{\al+n-1}{\be+m-n},
\]
implying, together with condition \eqref{mb2131}, that
\[
p_n=\frac{(m-n+1)\cdot(\al+n-1)}{n\cdot (\be+m-n)}\cdot  p_{n-1}
\]
with initial value
\[
p_0=\prod_{k=0}^{m-1}\frac{\be+k}{\al+\be+k}.
\]
\end{ex}

Similarly to the previous example, in the next one we treat the case $P_N={\bf MNB}\big(r,\rho_2(\vT)\big)$ ($r>0$). As a special instance, a well-known recursion for the generalized Waring distribution (cf., e.g., \cite{jkk}, Section 6.2.3, for the definition  and its properties) is rediscovered (see  \cite{hessa},  Example 4).

\begin{ex}
\label{MNB213}
\normalfont
Let  $P_N={\bf MNB}\big(r,\rho_2(\vT)\big)$, where $r>0$, and fix on arbitrary $n\in\N$. As $P_N\in {\bf Panjer}(a(\vT),b(\vT);0)$ with $a(\vT)=1-\rho_2(\vT)$ and $b(\vT)=(r-1)\cdot \big(1-\rho_2(\vT)\big)$ (see Table \ref{pt}), apply Theorem \ref{rec} to get
\begin{align}
p_{n}&=\frac{r+n-1}{n}\cdot\E_{\mu_{n-1}}\big[\big(1-\rho_2(\vT)\big)\big]\cdot p_{n-1}\nonumber\\
&=\frac{r+n-1}{n}\cdot\frac{\E_P\big[\big(\rho_2(\vT)\big)^r\cdot\big(1-\rho_2(\vT)\big)^{n}\big]}{\E_P\big[\big(\rho_2(\vT)\big)^r\cdot\big(1-\rho_2(\vT)\big)^{n-1}\big]}\cdot p_{n-1},
\label{mnb2131} 
\end{align}
while the initial value is given by condition $p_0=\E_P\big[\big(\rho_2(\vT)\big)^r\big]$.\smallskip
	
	In particular, if $D=(0,1)$, $\rho_2=\text{id}_D$ and $P_\vT={\bf Beta}(\alpha,\beta)$ for $\alpha,\beta>0$, then 
	\[
	\E_P[\vT^r\cdot(1-\vT)^{n}]=\frac{\mathrm{B}(\alpha+r,\beta+n)}{\mathrm{B}(\alpha,\beta)}
	\]
	implying, together with condition \eqref{mnb2131}, that
	\[
	p_{n}=\frac{(r+n-1)\cdot(\be+n-1) }{n\cdot(\al+\be+r+n-1)}\cdot p_{n-1}
	\]
	with initial value
	\[
	p_0=\frac{\mathrm{B}(\alpha+r,\beta)}{\mathrm{B}(\alpha,\beta)}.
	\]
\end{ex}

\section{Mixed compound distributions}\label{cd}

Let $X:=\{X_n\}_{n\in\N}$ be a sequence of random variables on $\vO$ with values in $R_X\subseteq\R_+$, and consider the  random variable $S$ on $\vO$ taking  values in $R_S\subseteq\R_+$, defined by 
\[
S:=\begin{cases}
	\sum_{j=1}^N X_j & \text{, if } N\geq 1\\
	0 &  \text{, if } N=0
\end{cases}.
\]
The sequence $X$ and the random variable $S$ are known as the \textbf{claim size process} and the \textbf{aggregate claims size}, respectively. Recall that a sequence $\{Z_n\}_{n\in\N}$ of random variables on $\vO$ is:
\begin{enumerate}
\item $P$-{\bf conditionally (stochastically) independent given $\vT$}, if for each $n\in\N$ with $n\geq 2$
\[
P\big(Z_{i_1}\leq z_{i_1}, Z_{i_2}\leq z_{i_2},\ldots, Z_{i_n}\leq z_{i_n}  \mid \vT\big)=\prod^{n}_{k=1} P\big(Z_{i_k}\leq z_{i_k}\mid\vT\big)\quad P{\uph}\sigma(\vT)\mbox{-a.s.}
\]
whenever $i_1,\ldots, i_n$ are distinct members of $\N$ and $(z_{i_1},\ldots,z_{i_n})\in\R^n$;
\smallskip
\item {\bf $P$-conditionally identically distributed given $\vT$}, if 
\[
P\big(F\cap Z_k^{-1}[B]\big)=P\big(F\cap Z_m^{-1}[B]\big)
\]
whenever $k,m\in \N$, $F\in\sigma(\vT)$ and $B\in \B$.
\end{enumerate} 

\textit{For the rest of the paper,  we simply write ``conditionally" in the place of ``conditionally given $\vT$” whenever conditioning refers to  $\vT$, and, unless otherwise stated, $X$ is a sequence of $P$-conditionally i.i.d random variables, and $P$-conditionally mutually independent of $N$. }\smallskip

Denote by $\mu$ either the Lebesgue measure $\lambda$ on $\B(0,\infty)$ or the counting measure on $\N_0$, and for any $\theta\in D$ let 
\[
(P_\theta)_{X_1}(B)=\int_B  f_\theta(x)\,\mu(dx)\quad\text{ for any } B\in\B(R_X),
\]
where $f_\theta:=\big(f_\theta\big)_{X_1}$ is the probability (density) function of $X_1$ with respect to $P_\theta$. Since $\{P_\theta\}_{\theta\in D}$ is an rcp of $P$ over $P_\vT$ consistent with $\vT$, it follows by \cite{lm6z3}, Lemma 3.2,  that the latter is equivalent to 
\[
P_{X_1\mid\vT}(B)=\int_B f_{X_1\mid\vT}(x)\,\mu(dx)\quad P{\uph}\sigma(\vT)\text{-a.s.}.
\]

\begin{rem}
	\label{remX}
	\normalfont
If $\vT$ is non-degenerate and if $\int\E^2_P[X_1\mid\vT]\,dP<\infty$, then  the random variables $X_n$, $X_m$ with $n\neq m$ are $P$-positively correlated.\smallskip
	
In fact, for any $n,m\in\N$ such that $n\neq m$, we get
	\[
	\textrm{Cov}_P(X_n,X_m)=\E_P\big[\E_P[X_n\cdot X_m\mid\vT]\big]  -\E^2_P[X_1]=\E_P\big[\E^2_P[X_1\mid\vT] \big]-\E^2_P\big[\E_P[X_1\mid\vT] \big]>0,
	\]
	where the first equality follows by \cite{mt2}, Remark 2.1, while the second one  follows since $X$ is $P$-conditionally i.i.d.; hence $\textrm{Cov}_P(X_n,X_m)=\text{Var}_P\big(\E_P[X_1\mid\vT]\big)> 0$. 
\end{rem}

The next proposition provides an expression for the probability distribution of a sum of $P$-conditionally i.i.d. random variables.

\begin{prop}
\label{con}
If $S_n:=\sum_{j=1}^nX_j$ for any $n\in\N$, then  
\[
P_{S_n}(B)=\int_B\Big(\int_{D} f^{\ast n}_\theta(x) \,P_\vT(d\theta)\Big)\,\mu(dx)\quad\text{for all } B\in\B(R_S),
\]
where $f_\theta^{\ast n}$ denotes the $n$-th convolution of $f_\theta$ with itself.
\end{prop}
\begin{proof} 
Fix on arbitrary  $n\in\N$. For any $B\in\B(R_S)$ and $F\in\B(D)$, the consistency of $\{P_\theta\}_{\theta\in D}$ with $\vT$, along with Lemma \ref{35}, yields
\[
P\big(S^{-1}_n[B]\cap\vT^{-1}[F]\big)=\int_{F}  (P_\theta)_{S_n} (B)\,P_\vT(d\theta).
\] 
But since $X$ is $P$-conditionally i.i.d.,  it follows by \cite{lm6z3},  Lemma 3.4(ii), that $X$ is $P_\theta$-i.i.d. for $P_\vT$-a.a. $\theta\in D$, implying that
\begin{gather}
(P_\theta)_{S_n} (B)=\int_B f_\theta^{\ast n}(x)\,\mu(dx) ;
\label{con1}
\end{gather}
hence for $F=D$, we get
\[
P_{S_n}(B)=\int_B\Big(\int_{D} f^{\ast n}_\theta(x) \,P_\vT(d\theta)\Big)\,\mu(dx)
\]
which completes the proof.
\end{proof}

The next example shows how  Proposition  \ref{con} can be utilized in order to compute convolutions in the case of $P$-conditionally i.i.d. random variables.

\begin{ex}
\label{par}
\normalfont
Let $D:=(0,\infty)$ and fix on arbitrary $n\in\N$. Assume that $P_{X_1\mid\vT}={\bf Exp}(\vT)$ $P{\uph}\sigma(\vT)$-a.s. and  $P_\vT={\bf Ga}(\al,\be)$ with $\alpha,\beta\in (0,\infty)$. Since by \cite{lm6z3}, Lemma 3.4(ii),  the claim  size process $X$ is $P_\theta$-independent and $(P_\theta)_{X_1}={\bf Exp}(\theta)$ for $P_\vT$-a.a. $\theta\in D$, we deduce that $(P_\theta)_{S_n}={\bf Ga}(n, \theta)$ for $P_\vT$-a.a. $\theta>0$. Now, by applying Proposition \ref{con}, we get  
\[
P_{S_n}(B)=\int_B\frac{\beta^\alpha\cdot x^{n-1}}{\Gamma(n) \cdot\Gamma(\alpha)}\cdot\Big(\int_{D}\theta^{\alpha+n-1}\cdot e^{-(\beta+x)\cdot \theta}\,\lambda(d\theta)\Big)\,\lambda(dx)
\]
for any $B\in\B(0,\infty)$, implying 
\[
P_ {S_n} (B)= \int_B\frac{\beta^\alpha}{\textrm{B}(\alpha,n)}\cdot\frac{x^{n-1}}{(\beta+x)^{\alpha+n}}\,\lambda(dx),
\]
i.e., $S_n$ is distributed according to the generalized Pareto distribution (cf., e.g., \cite{wil}, p. 118). Note that since each $X_n$ is $P$-conditionally exponentially distributed and $\vT$ is gamma distributed, it follows easily that each $X_n$ is Pareto distributed with parameters $\alpha$ and $\beta$  (cf., e.g., \cite{Sc}, p. 180 for the definition of the Pareto distribution). 
\end{ex}

The following definition is in accordance with the definition of mixed compound Poisson processes in \cite{lm3}, p. 780 and Lemma 3.3.

\begin{df}
\label{mcd}
\normalfont 
The aggregate claims distribution $P_S$ is  a {\bf mixed compound distribution}, written ${\bf MC}\big(P_{N},P_{X_1}\big)$ for brevity,  if $P_N$ is a mixed claim number distribution and the sequence $X$ is $P$-conditionally i.i.d. and $P$-conditionally mutually independent of $N$.
\end{df} 

In particular, if $P_\vT=\delta_{\theta_0}$ for some $\theta_0\in D$, then $P_S$ reduces to a \textit{compound distribution}, written ${\bf C}\big((P_{\theta_0})_N,(P_{\theta_0})_{X_1}\big)$ for short (cf., e.g., \cite{Sc},  p. 109).

\begin{rem}
\label{nx}
\normalfont
Assume that $\vT$ is a positive random variable and let $P_S={\bf MC}\big(P_{N},P_{X_1}\big)$. If the maps $\theta\mapsto\E_{P_\theta}[N]$ and $\theta\mapsto\E_{P_\theta}[X_1]$ have the same (resp. different) monotonicity for $P_\vT$-a.a. $\theta>0$, then the random variables $N$ and $X_1$ are $P$-positively (resp. negatively) correlated. \smallskip
	
In fact, assume first that the maps $\theta\mapsto\E_{P_\theta}[N]$ and $\theta\mapsto\E_{P_\theta}[X_1]$ have the same  monotonicity for $P_\vT$-a.a. $\theta>0$. By applying Lemma \ref{35} we get that $\E_P[N\cdot X_1]=\E_{P_\vT}\big[\E_{P_\theta}[N]\cdot\E_{P_\theta}[X_1]\big]$, implying together with \cite{Sc2014}, Theorem 2.2, and Lemma \ref{35}, that  $\E_P[N\cdot X_1]\geq\E_P[N]\cdot \E_P[X_1]$; hence $N$ and $X_1$ are positively correlated. The proof when  the maps $\theta\mapsto\E_{P_\theta}[N]$ and $\theta\mapsto\E_{P_\theta}[X_1]$ have different monotonicity, for $P_\vT$-a.a. $\theta>0$, is similar.
\end{rem}

The following characterization of mixed compound distributions  in terms of consistent rcp's allows us to eliminate conditioning via a suitable change of measures, and  to convert a mixed compound distribution  into a compound one with respect to the probability measures of the corresponding rcp. 
\begin{prop}
\label{mcd1}
The following statements are equivalent:
\begin{enumerate}
\item $P_S={\bf MC}\big(P_{N},P_{X_1}\big)$;
\smallskip
\item $(P_\theta)_S={\bf C}\big((P_{\theta})_N, (P_\theta)_{X_1}\big)$ for $P_\vT$-a.a. $\theta\in D$.
\end{enumerate} 
\end{prop}
\begin{proof}
Statement (i) is equivalent to the facts $P_{N\mid\vT}={\bf K}(\vT)$ $P{\uph}\sigma(\vT)$-a.s., $X$ is $P$-conditionally i.i.d. and $P$-conditionally mutually independent of $N$. Due to \cite{lm6z3}, Lemmas 3.2, 3.3 and 3.4(ii), we equivalently get that $(P_\theta)_{N}={\bf K}(\theta)$,  that  $N$ and $X$ are $P_\theta$-mutually independent and  that  $X$ is $P_\theta$-i.i.d., respectively, for $P_\vT$-a.a. $\theta\in D$, which is equivalent to (ii). 
\end{proof}

The characterization appearing in Proposition \ref{mcd1} allows  the extension of well-known results for compound distributions to their mixed counterpart. In the next result, we present a formula for the  probability distribution of $S$ in the case $P_S={\bf MC}\big(P_{N},P_{X_1}\big)$.

 \begin{cor}
 	\label{joint}
If $P_S={\bf MC}\big(P_{N},P_{X_1}\big)$, then
\[
P_S(B)=\sum_{n=0}^\infty \Big(\int_Dp_n(\theta)\cdot \int_Bf_\theta^{\ast n}(x)\,\mu(dx)\,P_\vT(d\theta)\Big)\quad\text{for any } B\in\B(R_S)
\]
  \end{cor}
\begin{proof}
Since  $P_S={\bf MC}\big(P_{N},P_{X_1}\big)$, apply Proposition \ref{mcd1} to get $(P_\theta)_S={\bf C}\big((P_{\theta})_N, (P_\theta)_{X_1}\big)$ for $P_\vT$-a.a. $\theta\in D$; hence by \cite{Sc}, Lemma 5.1.1, we obtain
\[
(P_\theta)_S(B)=\sum_{n=0}^\infty  p_n(\theta)\cdot \int_Bf_\theta^{\ast n}(x)\,\mu(dx)\quad\text{for any } B\in\B(R_S). 
\]
The latter, together with Lemma \ref{35}, completes the proof.
\end{proof}

Denote by $g$ the probability (density) function of the random variable $S$ with respect to $P$. 
The next example presents an explicit formula for the tail probability $P(S>u)$ in the case of a mixed compound geometric distribution with  conditionally exponentially distributed  claim sizes.

\begin{ex}
\label{comp}
\normalfont
Let  $P_S={\bf MC}\big(P_{N},P_{X_1}\big)$ with $P_N=\textbf{MNB}\big(1,\rho_2(\vT)\big)$  and    $P_{X_1\mid\vT}=\textbf{Exp}\big(v(\vT)\big)$ $P{\uph}\sigma(\vT)$-a.s.,  where $v$ is a  $\B(D)$-$\B(0,\infty)$-measurable function. By Proposition \ref{mcd1}, we get that $(P_\theta)_S={\bf C}\big((P_{\theta})_N, (P_\theta)_{X_1}\big)$  with $(P_{\theta})_N=\textbf{NB}\big(1, \rho_2(\theta)\big)$ and $(P_\theta)_{X_1}=\textbf{Exp}\big(v(\theta)\big)$ for $P_\vT$-a.a. $\theta\in D$. An easy computation yields 
\[
g(x)=\begin{cases}
\E_P\big[\rho_2(\vT)\big]	 & \text{, if } x=0\\
\E_P\Big[\big(1-\rho_2(\vT)\big)\cdot \rho_2(\vT)\cdot  v(\vT)\cdot e^{-\rho_2(\vT)\cdot v(\vT)\cdot x}\Big]  & \text{, if } x>0
\end{cases};
\]
hence 
\begin{gather}
P(S>u)=\E_P\Big[\big(1-\rho_2(\vT)\big)\cdot e^{-\rho_2(\vT)\cdot v(\vT)\cdot u}\Big]\quad\text{for any } u>0.
\label{381}
\end{gather}
	
\noindent In particular,\smallskip 
	
\noindent{\textbf{(a)}} if  $D=(0,\infty)$, $\rho_2(\vT) = \frac{1}{1+\vT}$ and $v(\vT)=\vT^2$, then condition \eqref{381} can be rewritten as
\[
P(S>u)=\E_P\Big[ \frac{\vT}{1+\vT} \cdot e^{-\frac{\vT^2}{1+\vT} \cdot u} \Big] \quad \text{for any } u > 0.
\]
Figure \ref{fig38a} displays the tail probability $P(S>u)$ in the case that the random variable $\vT$ is distributed according to the Inverse Gaussian distribution with parameters $\mu, \varphi>0$ (written \textbf{IG}$(\mu,\varphi)$ for short), i.e., 
\[
P_\vT(B)=\int_B \sqrt\frac{\varphi}{2 \pi x^3}\cdot e^{-\frac{\varphi (x-\mu)^2}{2 \mu^2 x}}\,\lambda(dx)\quad\text{for any } B\in\B(0,\infty)
\] 
\begin{figure}[H]
\centering
\begin{subfigure}[t]{0.35\textwidth}
\subcaption*{\textbf{IG}$(1,\varphi)$}
\includegraphics[width=1\linewidth]{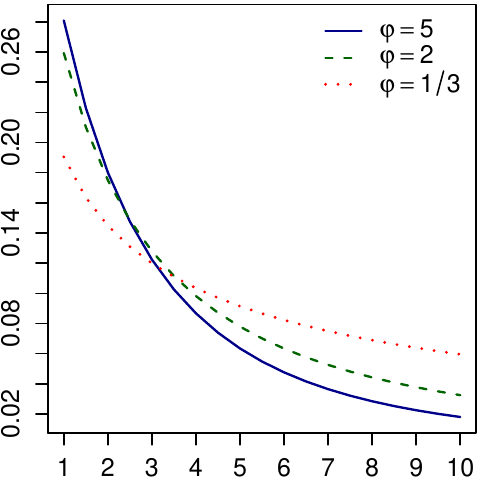}	
\label{fig38a1}
\end{subfigure}\hspace{8mm} 
~ 
\begin{subfigure}[t]{0.35\textwidth}
\subcaption*{\textbf{IG}$(\mu,1)$}
\includegraphics[width=1\linewidth]{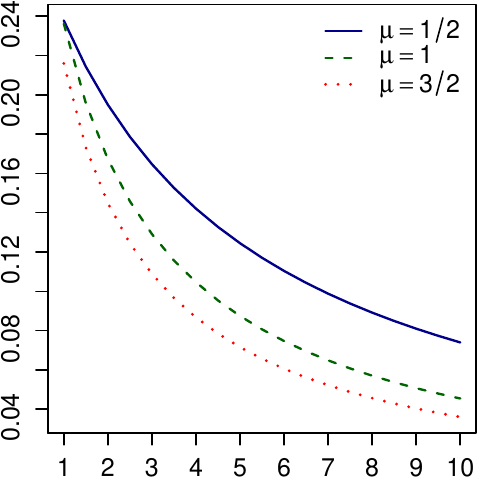}
\label{fig38a2}
\end{subfigure}
\caption{The tail probability $P(S>u)$ for different values of $\mu$ and $\varphi$.}
\label{fig38a}
\end{figure}
\noindent Note that this particular selection for $\rho_2(\vT)$ and $P_\vT$ has also been adopted in a regression context by \cite{tzougas}; \smallskip 
	
\noindent{\textbf{(b)}} Analogously, if $\rho_2(\vT)=e^{-\vT}$, $v(\vT)=\vT^3$, then condition \eqref{381} becomes
\[
P(S>u)=\E_P\Big[\big(1-e^{-\vT}\big)\cdot e^{-e^{-\vT}\cdot \vT^3 \cdot u}\Big] \quad \text{for any } u > 0.
\]
Figure \ref{fig38b} depicts the tail probability $P(S>u)$ in the case $P_\vT=\textbf{Ga}(\al,\be)$, where $\al,\be>0$. 
\begin{figure}[H]
\centering
\begin{subfigure}[t]{0.35\textwidth}
\subcaption*{\textbf{Ga}$(3,\beta)$}
\includegraphics[width=1\linewidth]{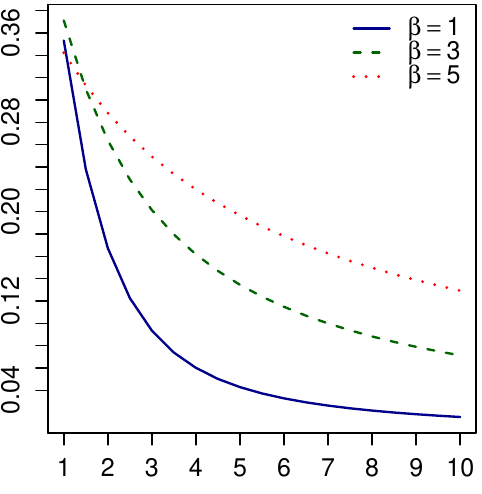}
\label{fig38b1}
\end{subfigure}\hspace{8mm} 
~ 
\begin{subfigure}[t]{0.35\textwidth}
\subcaption*{\textbf{Ga}$(\alpha,3)$}
\includegraphics[width=1\linewidth]{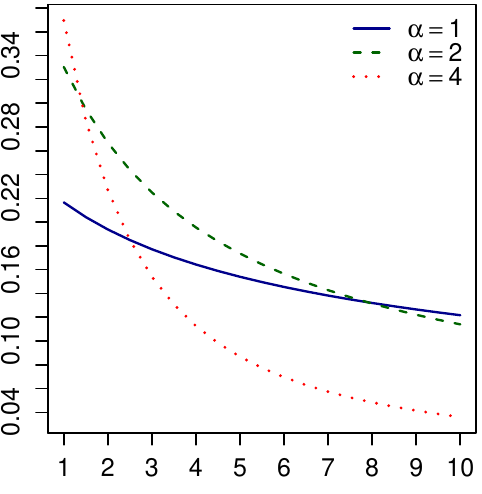}
\label{fig38b2}
\end{subfigure}
\caption{The tail probability $P(S>u)$ for different values of $\al$ and $\be$.}
\label{fig38b}
\end{figure}
\noindent Note that a similar mixture for $N$ has been studied by \cite{gy}; \smallskip
	
\noindent{\textbf{(c)}} Now, let $D=(0,1)$,  $\rho_2(\vT)=\vT$ and $v(\vT)=-\ln\vT^2$. In this case, condition \eqref{381} can be rewritten as
\[
P(S>u)=\E_P\Big[\big(1-\vT\big)\cdot \vT^{2\cdot u\cdot \vT}\Big] \quad \text{for any } u > 0.
\]
Figure \ref{fig38c} illustrates the tail probability $P(S>u)$ in the case $P_\vT=\textbf{Be}(\al,\be)$, where $\al,\be>0$. Note that such a mixture for $N$ has been considered by \cite{zwang}. 
\begin{figure}[H]
\centering
\begin{subfigure}[t]{0.35\textwidth}
\subcaption*{\textbf{Be}$(3,\beta)$}
\includegraphics[width=1\linewidth]{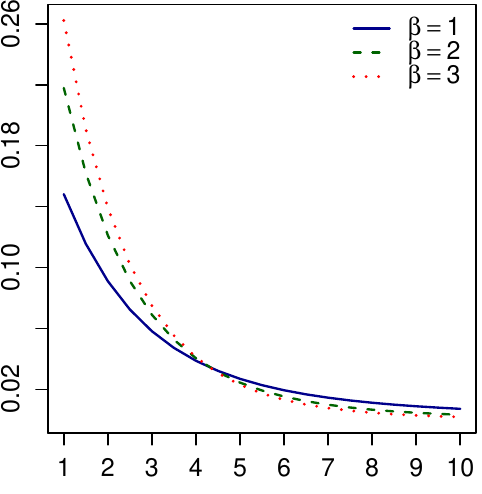}
\label{fig38c1}
\end{subfigure}\hspace{8mm} 
~ 
\begin{subfigure}[t]{0.35\textwidth}
\subcaption*{\textbf{Be}$(\alpha,3)$}
\includegraphics[width=1\linewidth]{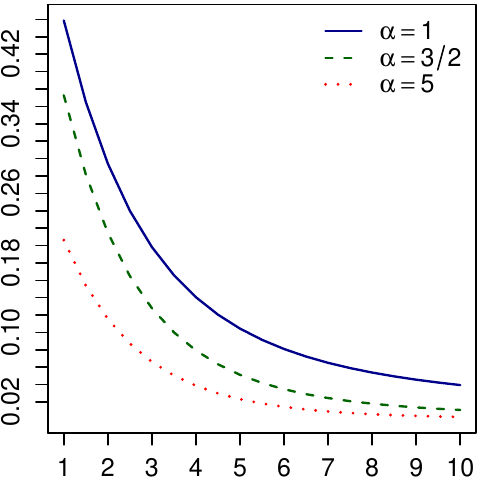}
\label{fig38c2}
\end{subfigure}
\caption{The tail probability $P(S>u)$ for different values of $\al$ and $\be$.}
\label{fig38c}
\end{figure}
\end{ex}

\section{Recursive formulas for mixed compound distributions}\label{rcd}

\textit{Throughout what follows in this section, unless stated otherwise, $P_S={\bf MC}\big(P_{N},P_{X_1}\big)$ with $P_N\in  \text{\bf Panjer}(a(\vT),b(\vT);0)$ and $P_{X_1}(\N_0)=1$.}\smallskip

Since in the collective risk model the random variable $S$ represents the total claim amount paid out from an insurance company over a fixed period of time, the determination of its pmf is of great importance. However, even in the classical case where $X$ is $P$-i.i.d. and $N, X$ are $P$-mutually independent, it is difficult to obtain a closed form for the pmf of $S$, which usually involves  the computation of higher-order convolutions, a task that may be proven difficult, or at least time consuming. As a consequence, actuaries usually resort to recursive expressions, which seem to be more convenient and computationally efficient.\smallskip

In the next result,  we obtain a recursive formula for the pmf of $S$ by extending the famous Panjer recursion (cf., e.g., \cite{Sc}, Theorem 5.4.2) to the class $\text{\bf Panjer}(a(\vT),b(\vT);0)$. In order to present it, we denote by $m_{N\mid\vT}$ the conditional probability generating functions of $N$, i.e., $m_{N\mid\vT}(t)=\E_P\big[t^N\mid\vT\big]$ for any $t$ with $|t|\leq1$.

\begin{thm}
	\label{pan2}
	There exists a family  $\{\rho_x\}_{x\in{R}_S}$ of probability measures on $\vS$ defined by  
	\[
	\rho_x(A):=\frac{\E_P[\1_A\cdot\1_{\{S=x\}}]}{\E_P[\1_{\{S=x\}}]}\quad\text{for any } A\in\vS
	\] 
	being an rcp of $P$ over $P_S$ consistent with $S$, such that 
	\[
	g(x)=\begin{cases}
		\E_P\big[m_{N\mid\vT} \big(f_{X_1\mid\vT}(0)\big)\big]	 & \text{, if } x=0\\
		\sum_{y=1}^xD_{x,y}\cdot g(x-y)& \text{, if } x\in R_S\setminus\{0\}
	\end{cases},
	\]
	where 
	\[
	D_{x,y}:=\E_{\rho_{x-y}}\bigg[\Big( a(\vT)+\frac{b(\vT)\cdot y}{x}\Big)\cdot \frac{f_{X_1\mid\vT}(y)}{1-a(\vT)\cdot f_{X_1\mid\vT}(0)}\bigg]
	\]
	for any $x\in R_S\setminus\{0\}$ and $y\in\{1,\ldots,x\}$.\medskip
	
	In particular, if $f_{X_1\mid\vT}(0)=0$ $P{\uph}\sigma(\vT)$-a.s. then
	\[
	g(x)=\begin{cases}
		p_0	 & \text{, if } x=0\\
		\sum_{y=1}^xD_{x,y}\cdot g(x-y)& \text{, if } x\in R_S\setminus\{0\}
	\end{cases}, 
	\]
	where 
	\[
	D_{x,y}=\E_{\rho_{x-y}}\Big[\Big( a(\vT)+\frac{b(\vT)\cdot y}{x}\Big)\cdot \ f_{X_1\mid\vT}(y)\Big]
	\]
	for any $x\in R_S\setminus\{0\}$ and $y\in\{1,\ldots,x\}$. 
\end{thm}
\begin{proof}
	Fix on arbitrary $x\in R_S$ and define the set-function $\rho_x:\vS\to(0,\infty)$ by 
	\[
	\rho_x(A):=\frac{\E_P[\1_A\cdot\1_{\{S=x\}}]}{\E_P[\1_{\{S=x\}}]}\quad\text{for any } A\in\vS.
	\] 		
By applying similar arguments to those appearing in the proof of Theorem \ref{rec}, it follows easily that the family $\{\rho_y\}_{y\in{R}_S}$ is an rcp of $P$ over $P_S$ consistent with $S$.  Furthermore, first note that condition $g(0)=\E_P\big[m_{N\mid\vT}\big(f_{X_1\mid\vT}(0)\big)\big]$  is an immediate consequence of Corollary \ref{joint} for $B=\{0\}$.  By \cite{lm6z3}, Lemmas 3.4(ii) and 3.3,  the sequence  $X$ is $P_\theta$-i.i.d. and $(N,X)$ are $P_\theta$-mutually independent for $P_\vT$-a.a. $\theta\in D$, respectively. Since $P_N\in {\bf Panjer}(a(\vT),b(\vT);0)$,  Proposition \ref{mpanchar}  implies that $(P_\theta)_N\in {\bf Panjer}(a(\theta),b(\theta);0)$ for $P_\vT$-a.a. $\theta\in D$; hence we may apply  \cite{Sc}, Theorem 5.4.2, to get 
	\[
	g_\theta(x)= \frac{1}{1-a(\theta)\cdot f_\theta(0)}\cdot\sum_{y=1}^x\Big( a(\theta)+\frac{b(\theta)\cdot y}{x}\Big)\cdot f_\theta(y)\cdot g_\theta(x-y),
	\]
where $g_\theta$ denotes the pmf of $S$ with respect to $P_\theta$. The latter, together with Corollary \ref{joint} and Lemma \ref{35}, yields 
	\begin{equation}
		g(x)=\sum_{y=1}^x\E_P\Big[\Big( a(\vT)+\frac{b(\vT)\cdot y}{x}\Big)\cdot \frac{f_{X_1\mid\vT}(y)}{1-a(\vT)\cdot f_{X_1\mid\vT}(0)}\cdot g_\vT(x-y)\Big].
		\label{391}
	\end{equation}
But since  
	\begin{align*}
		\E_{\rho_x}[\1_{\vT^{-1}[F]}]=\frac{\E_P[\1_{\vT^{-1}[F]}\cdot\1_{\{S=x\}}]}{\E_P[\1_{\{S=x\}}]}=\frac{\E_{P_\vT}[\1_F\cdot g_{\theta}(x)]}{g(x)}\quad\text{for each } F\in\B(D),
	\end{align*} 
	where the second equality follows by Lemma \ref{35}, we easily get that  
	\begin{gather}
		\begin{split}
			&\E_P\Big[\Big( a(\vT)+\frac{b(\vT)\cdot y}{x}\Big)\cdot \frac{f_{X_1\mid\vT}(y)}{1-a(\vT)\cdot f_{X_1\mid\vT}(0)}\cdot g_\vT(x-y)\Big]\\ &\quad=\E_{\rho_{x-y}}\Big[\Big( a(\vT)+\frac{b(\vT)\cdot y}{x}\Big)\cdot \frac{f_{X_1\mid\vT}(y)}{1-a(\vT)\cdot f_{X_1\mid\vT}(0)}\Big]\cdot g(x-y)
			\label{392}
		\end{split}
	\end{gather}
	for any $y\in\{1,\ldots,x\}$. Conditions \eqref{391} and \eqref{392} complete  the first part of the proof. \smallskip
	
	In particular, first note that if $f_{X_1\mid\vT}(0)=0$ $P{\uph}\sigma(\vT)$-a.s., then $m_{N\mid\vT}\big(f_{X_1\mid\vT}(0)\big)=p_0(\vT)$ $P{\uph}\sigma(\vT)$-a.s.. The latter, along with conditions \eqref{391} and \eqref{392}, completes the proof. 
\end{proof}

\begin{rem}\normalfont
	Assume that $P_{X_1}$ is absolutely continuous with respect to the Lebesgue measure $\lambda$ on $\B(0,\infty)$. By applying the characterization appearing on Proposition \ref{mcd1} in conjunction with \cite{pa}, Theorem on p. 24, we can easily prove the equality
	\begin{gather}
			g(x) =\E_{P}\big[p_1(\vT)\cdot f_{X_1\mid\vT}(x)\big] + \int_0^x\E_{P}\bigg[\Big(\alpha(\vT)+\frac{b(\vT)\cdot y}{x}\Big)\cdot f_{X_1\mid\vT}(y)\cdot g_\vT(x-y)\bigg]\,\lambda(dy)
		\label{pancon}
	\end{gather}
	for any  $x>0$. The integral equation \eqref{pancon} has to be solved numerically, however its implementation presents some serious difficulties (see \cite{gdsco}, p. 42).  The proof of a recursion for $g$ in the spirit of Theorem \ref{pan2} and its implementation is the subject of a forthcoming paper. 
\end{rem}

\begin{rem}
	\label{com}
	\normalfont
If $X$ and  $\vT$ are $P$-independent, it  follows by Lemma \ref{35} that $P_{X_1\mid\vT}=P_{X_1}$ $P{\uph}\sigma(\vT)$-a.s.. Hence   Theorem  \ref{pan2} gives
	\[
	g(x)=\begin{cases}
		\E_P\Big[m_{N\mid\vT}\big(f_{X_1}(0)\big)\Big]	 & \text{, if } x=0\\
		\sum_{y=1}^x{D}_{x,y}\cdot g(x-y)& \text{, if } x\in R_S\setminus\{0\}
	\end{cases},
	\]
	where ${D}_{x,y}=\E_{\rho_{x-y}}\Big[\frac{x\cdot a(\vT)+y\cdot b(\vT)}{x\cdot(1-a(\vT)\cdot f_{X_1}(0))}\Big]\cdot f_{X_1}(y)$ for any $x\in R_S\setminus\{0\}$ and $y\in\{1,\ldots,x\}$.
\end{rem}

Recall that a sequence $Z:=\{Z_n\}_{n\in\N}$ of random variables on $\vO$ is called \textit{exchangeable}, if the joint distribution of the sequence is invariant under the permutation of the indices (cf., e.g., \cite{fr4}, 459C). In non-classical Risk Theory, the concept of exchangeability seems to be an appealing way to introduce a dependence structure between the claim sizes. Actually, exchangeability can be a natural assumption in the risk model context (see \cite{acl}, Remark 2.7). Note that  by \cite{lm6z3}, Proposition 4.1, we get that a sequence $Z$ is  exchangeable under $P$ if and only if there exists a  random vector $\widetilde\vT$ such that $Z$ is $P$-conditionally i.i.d. over $\sigma(\widetilde\vT)$. 

\begin{cor}
	\label{exch}
	Let $\{\rho_x\}_{x\in{R}_S}$ be the family  of probability measures constructed in Theorem \ref{pan2}. If $P_N\in{\bf Panjer}(a,b;0)$ and $N$, $\vT$ are $P$-independent then
	\[
	g(x)=\begin{cases}
		\E_P\Big[m_N \big(f_{X_1\mid\vT}(0)\big)\Big]	 & \text{, if } x=0\\
		\sum_{y=1}^x D_{x,y}\cdot g(x-y) & \text{, if } x\in R_S\setminus\{0\}
	\end{cases},
	\]
	where $D_{x,y}=\big( a +b \cdot \frac{y}{x}\big)\cdot\E_{\rho_{x-y}}\Big[\frac{f_{X_1\mid\vT}(y)}{1-a \cdot f_{X_1\mid\vT}(0)}\Big]$ for any $x\in R_S\setminus\{0\}$ and $y\in\{1,\ldots,x\}$.
\end{cor}

\noindent The proof follows immediately by Theorem \ref{pan2} and therefore it is omitted.\smallskip 

The following example shows that if $P_N\in {\bf Panjer}(a(\vT),b(\vT);0)$, then the claim number distribution, for example in the case of an Excess of Loss reinsurance, remains in the same class, but with different parameters (cf., e.g., \cite{Sc}, p. 108).

\begin{ex}
\label{thin}
\normalfont
Let  $h$ be a $\B(D)$-$\B(0,\infty)$-measurable function and set $v(\theta):=P_\theta(X_1>h(\theta))\in(0,1)$ for any $\theta\in D$. Fix  an arbitrary $n\in\N$ and define the function $z_n:\vO\times D\rightarrow\{0,1\}$ by  $z_n(\omega,\theta):=\1_{\{X_n>h(\theta)\}}(\omega)$ for any $(\omega,\theta)\in\vO\times D$. For an arbitrary, but fixed $\theta\in D$, set $z(\theta):=\{z_k(\theta)\}_{k\in\N}$. Define the real-valued function $Z_n(\vT)$ on $\vO$ by  $Z_n(\vT):=z_n\circ (id_\vO\times \vT)$ and the sequence $Z(\vT):=\{Z_k(\vT)\}_{k\in\N}$. Clearly $Z(\vT)$ is $P$-conditionally i.i.d. and $P$-conditionally independent of $N$ with  $P_{Z_n(\vT)}={\bf MB}(1,v(\vT))$. Consider the random sum $M:=\sum_{k=1}^N Z_k(\vT)$. It then follows by Theorem \ref{pan2} that 
	\[
	g(x)=\begin{cases}
		\E_P\big[m_{N\mid\vT}\big(1-v(\vT)\big)\big] & \text{, if } x=0\\
		D_{x,1}\cdot g(x-1) & \text{, if } x\in R_M\setminus\{0\}
	\end{cases},
	\]
	where 
 \[
 D_{x,1}=\E_{\rho_{x-1}}\bigg[\Big(\frac{a(\vT)\cdot v(\theta)}{1-a(\vT)\cdot\big(1-v(\vT)\big) }+\frac{1}{x} \cdot\frac{b(\vT)\cdot v(\vT)}{1-a(\vT)\cdot\big(1-v(\vT)\big)}\Big)\bigg].
 \]

\medskip
\noindent The latter, along with the fact that $P_M$ is a claim number distribution (cf., e.g., \cite{Sc}, Theorem 5.1.4), implies that $P_M\in {\bf Panjer}(\widetilde{a}(\vT),\widetilde{b}(\vT);0)$, where $\widetilde{a}(\vT):=\frac{a(\vT)\cdot v(\vT)}{1-a(\vT)\cdot\big(1-v(\vT)\big) }$ and $\widetilde{b}(\vT):= \frac{b(\vT)\cdot v(\vT)}{1-a(\vT)\cdot\big(1-v(\vT)\big)}$. 
\end{ex}

The following numerical examples demonstrate the behaviour of the pmf of the aggregate claims $ S $. For illustrative purposes, we focus on three widely used mixing distributions in actuarial literature, namely the Inverse Gaussian (IG), the the Gamma (Ga) and the Beta (Be) distribution.

\begin{ex} 
	\label{expcpp}
	\normalfont
	Let $D\subseteq(0,\infty)$ and $P_N=\textbf{MP}\big(\xi(\vT)\big)$, where $\xi$ is a $\B(D)$-$\B(0,\infty)$-measurable function. Since $P_N\in {\bf Panjer}(a(\vT),b(\vT);0)$, we get $\al(\vT)=0$ and $b(\vT)=\xi(\vT)$ (see Table \ref{pt}). Let $v$ be a $\B(D)$-$\B(0,1)$-measurable function. \smallskip 
	
	\noindent \textbf{(a)}  Take $P_{X_1}=\textbf{MNB}\big(1,v(\vT)\big)$. Since $P_{X_1}(\N_0)=1$, we may apply Theorem \ref{pan2} to get 
	\[
	g(x)=\begin{cases}
		\E_P\big[e^{\xi(\vT)\cdot (v(\vT)-1)}\big]	 & \text{, if } x=0\\
		\sum_{y=1}^xD_{x,y}\cdot g(x-y)& \text{, if } x\in \N  
	\end{cases}, 
	\]
	where $D_{x,y}=\frac{y}{x}\cdot\E_{\rho_{x-y}}\Big[\xi(\vT)\cdot v(\vT)\cdot\big(1-v(\vT)\big)^y\Big]$ for any $x\in\N$ and $y\in\{1,\ldots,x\}$. Figure \ref{fig46a} depicts the behaviour of $g$ for $\xi(\vT)=\vT$ and $v(\vT)=\frac{\vT}{1+\vT}$.
	\begin{figure}[H]
		\centering
		\begin{subfigure}[t]{0.32\textwidth}
			\subcaption*{\textbf{IG}$(2,5)$}
			\includegraphics[width=0.95\linewidth]{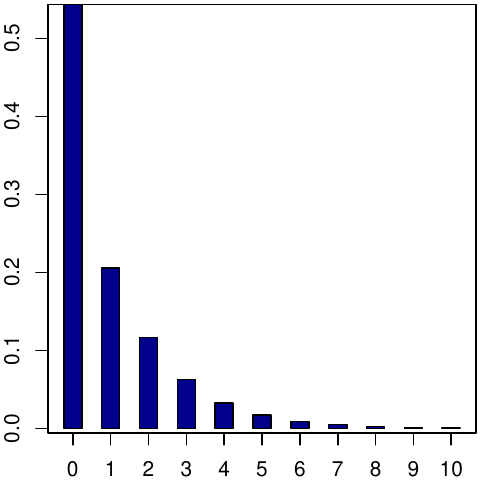}
			\label{fig46a1}
		\end{subfigure}
		~ 
		\begin{subfigure}[t]{0.32\textwidth}
			\subcaption*{\textbf{Ga}$(5,4)$}
			\includegraphics[width=0.95\linewidth]{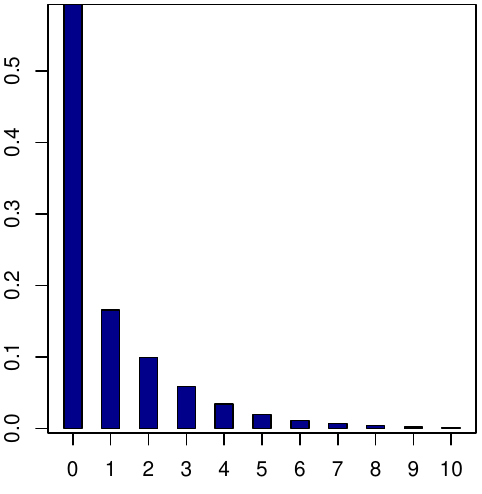}
			\label{fig46a2}
		\end{subfigure}
		~ 
		\begin{subfigure}[t]{0.32\textwidth}
			\subcaption*{\textbf{Be}$(5,5)$}
			\includegraphics[width=0.95\linewidth]{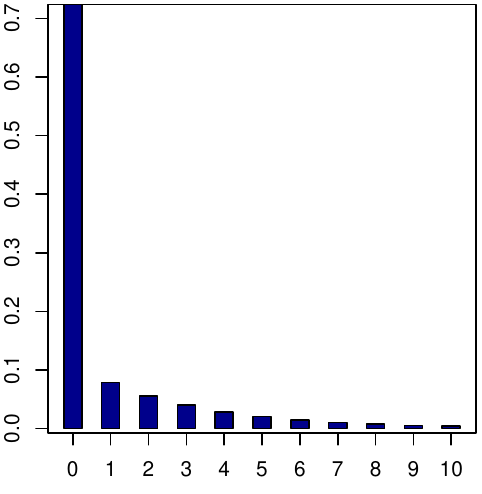}
			\label{fig46a3}
		\end{subfigure}
		\caption{The pmf of $ S $  for $P_N=\textbf{MP}(\vT)$ and  $P_{X_1}=\textbf{MNB}\big(1,\frac{\vT}{1+\vT}\big)$.}
		\label{fig46a}
	\end{figure}
	
	\noindent \textbf{(b)} Assume now that the claim sizes are distributed according to a mixed zero truncated geometric distribution with parameter $v(\vT)$ (written $\textbf{MZTG}\big(v(\vT)\big)$ for short), i.e., 
	\[
f_{X_1\mid\vT}(x) = v(\vT)\cdot\big(1-v(\vT)\big)^{x-1}\quad P{\uph}\sigma(\vT)\text{-a.s.}
	\]
	for any $x\in\N$. Since   $P_{X_1}(\N_0)=1$ with $f_{X_1\mid\vT}(0)=0$ $P{\uph}\sigma(\vT)$-a.s.,  we may apply Theorem~\ref{pan2} to get 
	\[
	g(x)=\begin{cases}
		\E_P\big[e^{-\xi(\vT)}\big]	 & \text{, if } x=0\\
		\sum_{y=1}^xD_{x,y}\cdot g(x-y)& \text{, if } x\in \N
	\end{cases},
	\]
	where $D_{x,y}=\frac{y}{x}\cdot\E_{\rho_{x-y}}\big[\xi(\vT)\cdot v(\vT)\cdot\big(1-v(\vT)\big)^{y-1}\big]$ for any $x\in\N$ and $y\in\{1,\ldots,x\}$. Figure \ref{fig46b} illustrates the behaviour of $g$ for the same choices of $\xi(\vT)$ and $v(\vT)$ as in (a).
	\begin{figure}[H]
		\centering
		\begin{subfigure}[t]{0.32\textwidth}
			\subcaption*{\textbf{IG}$(2,5)$}
			\includegraphics[width=0.95\linewidth]{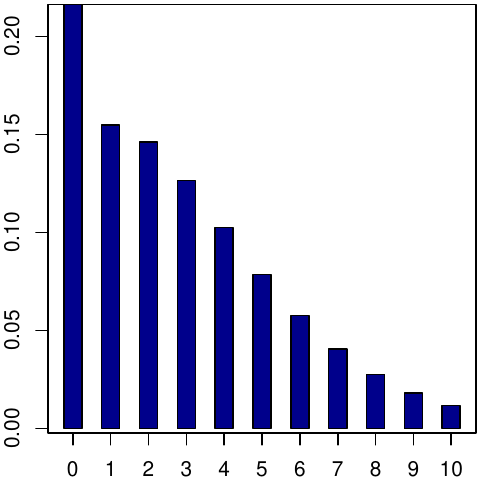}
			\label{fig46b1}
		\end{subfigure}
		~ 
		\begin{subfigure}[t]{0.32\textwidth}
			\subcaption*{\textbf{Ga}$(5,4)$}
			\includegraphics[width=0.95\linewidth]{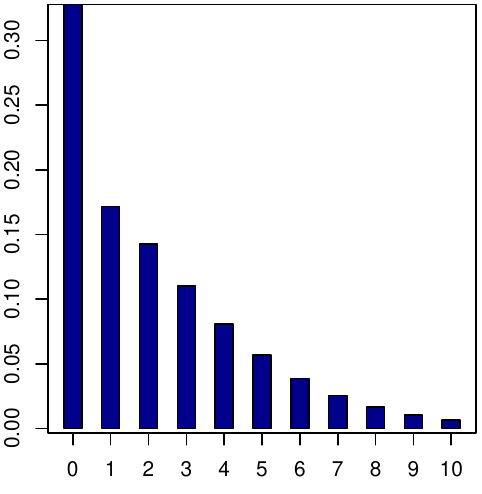}
			\label{fig46b2}
		\end{subfigure}
		~ 
		\begin{subfigure}[t]{0.32\textwidth}
			\subcaption*{\textbf{Be}$(5,5)$}
			\includegraphics[width=0.95\linewidth]{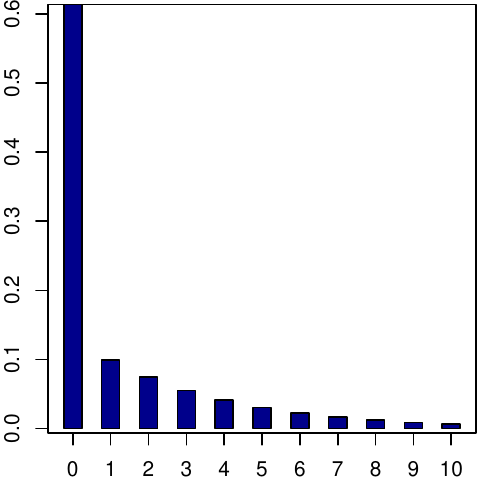}
			\label{fig46b3}
		\end{subfigure}
		\caption{The pmf of $ S $ for $P_N=\textbf{MP}(\vT)$ and $P_{X_1}=\textbf{MZTG}\big(\frac{\vT}{1+\vT}\big)$.}
		\label{fig46b}
	\end{figure}
	
	\noindent \textbf{(c)} Assume that  $X$ and $\vT$ are $P$-independent and that $v(\vT)$ is degenerate at some point $p\in(0,1)$. If $P_{X_1}=\textbf{ZTG}(p)$, then we can apply Theorem \ref{pan2}, along with  Remark \ref{com}(a), to get 
	\[
	g(x)=\begin{cases}
		\E_P\big[e^{-\xi(\vT)}\big]	 & \text{, if } x=0\\
		\sum_{y=1}^x{D}_{x,y} \cdot g(x-y)& \text{, if } x\in \N
	\end{cases},
	\]
	where ${D}_{x,y}=\frac{y}{x}\cdot p\cdot (1-p)^{y-1}\cdot\E_{\rho_{x-y}}\big[\xi(\vT)\big]$ for any $x\in\N$ and $y\in\{1,\ldots,x\}$. Figure \ref{fig46c} depicts  $g$ for $p=3/5$ and $\xi(\vT)$ as in (a).
	\begin{figure}[H]
		\centering
		\begin{subfigure}[t]{0.32\textwidth}
			\subcaption*{\textbf{IG}$(2,5)$}
			\includegraphics[width=0.95\linewidth]{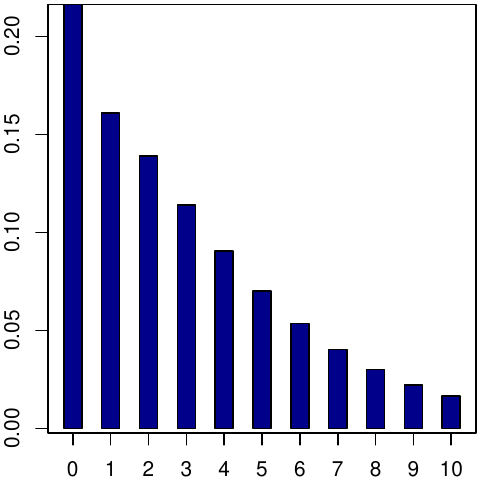}
			\label{fig46c1}
		\end{subfigure}
		~ 
		\begin{subfigure}[t]{0.32\textwidth}
			\subcaption*{\textbf{Ga}$(5,4)$}
			\includegraphics[width=0.95\linewidth]{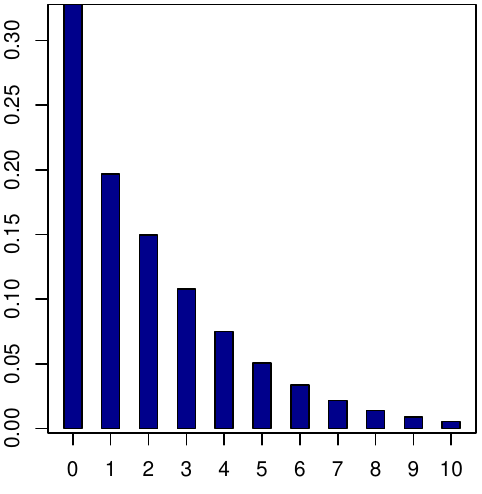}
			\label{fig46c2}
		\end{subfigure}
		~ 
		\begin{subfigure}[t]{0.32\textwidth}
			\subcaption*{\textbf{Be}$(5,5)$}
			\includegraphics[width=0.95\linewidth]{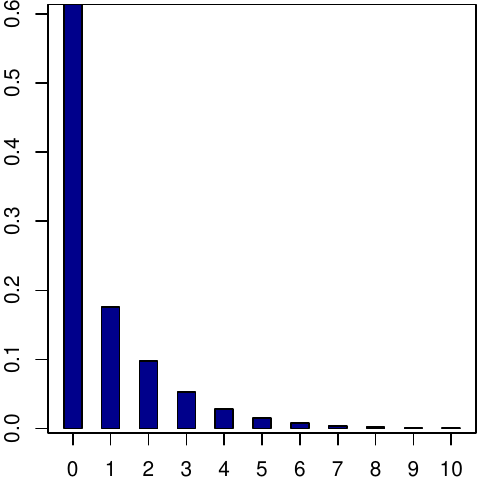}
			\label{fig46c3}
		\end{subfigure}
		\caption{The pmf of $ S $ for $P_N=\textbf{MP}(\vT)$ and $P_{X_1}=\textbf{ZTG}\big(\frac{3}{5}\big)$.}
		\label{fig46c}
	\end{figure}
	
	\noindent \textbf{(d)} Assume that $N$ and $\vT$ are $P$-independent and that $\xi(\vT)$ is degenerate at some point  $\theta_0\in D$. If $P_{X_1}=\textbf{MZTG}\big(v(\vT)\big)$, we may apply Corollary \ref{exch} to get 
	\[
	g(x)=\begin{cases}
		e^{-\theta_0}  & \text{, if } x=0\\
		\sum_{y=1}^x {D}_{x,y} \cdot g(x-y)& \text{, if } x\in \N
	\end{cases}, 
	\]
	where ${D}_{x,y}=\frac{y\cdot \theta_0}{x}\cdot\E_{\rho_{x-y}}\big[v(\vT)\cdot\big(1-v(\vT)\big)^{y-1}\big]$ for any  $x\in \N$ and $y\in\{1,\ldots x\}$. Figure \ref{fig46d} exhibits $g$ for $\theta_0=3$ and $v(\vT)$ as in (a).
	\begin{figure}[H]
		\centering
		\begin{subfigure}[t]{0.32\textwidth}
			\subcaption*{\textbf{IG}$(2,5)$}
			\includegraphics[width=0.95\linewidth]{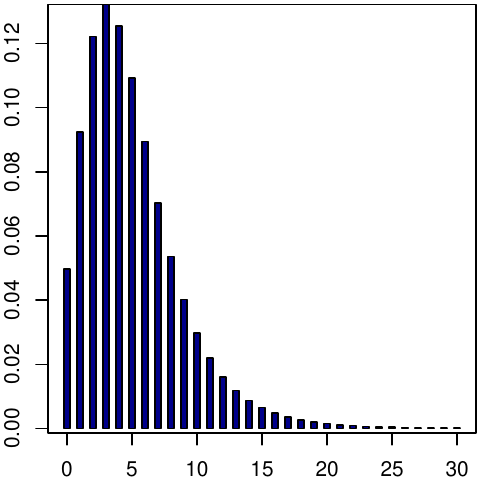}
			\label{fig46d1}
		\end{subfigure}
		~ 
		\begin{subfigure}[t]{0.32\textwidth}
			\subcaption*{\textbf{Ga}$(5,4)$}
			\includegraphics[width=0.95\linewidth]{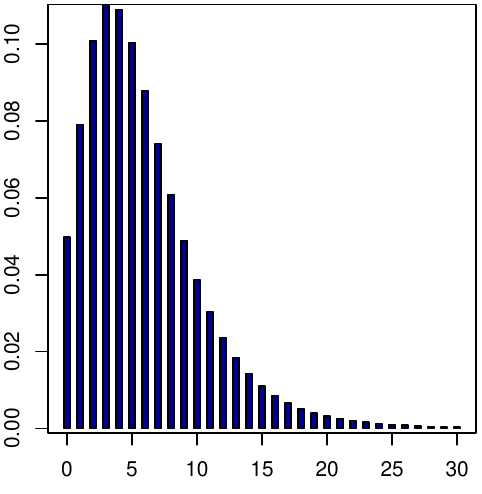}
			\label{fig46d2}
		\end{subfigure}
		~ 
		\begin{subfigure}[t]{0.32\textwidth}
			\subcaption*{\textbf{Be}$(5,5)$}
			\includegraphics[width=0.95\linewidth]{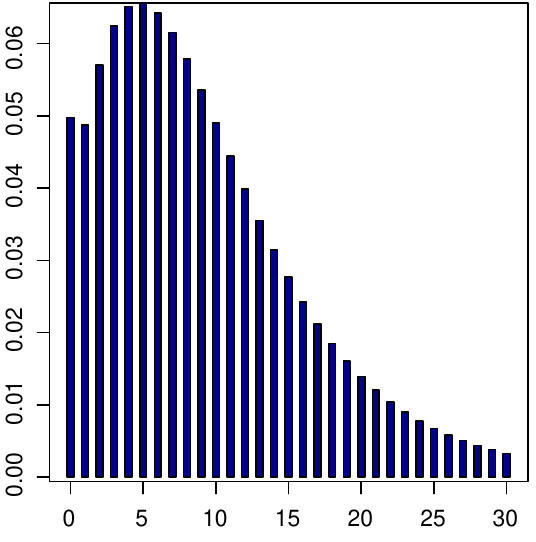}
			\label{fig46d3}
		\end{subfigure}
		\caption{The pmf of $ S $ for $P_N=\textbf{P}(3)$ and $P_{X_1}=\textbf{MZTG}\big(\frac{\vT}{1+\vT}\big)$.}
		\label{fig46d}
	\end{figure}
\end{ex}

\begin{ex}
\label{expcnb1}
\normalfont
Let $D\subseteq (0,\infty)$, $P_N=\textbf{MNB}\big(\rho_1(\vT),\rho_2(\vT)\big)$, where  $\rho_1, \rho_2$ are $\B(D)$-$\B(0,\infty)$- and $\B(D)$-$\B(0,1)$-measurable functions, respectively. Since $P_N\in {\bf Panjer}(a(\vT),b(\vT);0)$, we get $\al(\vT)=1-\rho_2(\vT)$ and $b(\vT)=\big(\rho_1(\vT)-1\big)\cdot\big(1-\rho_2(\vT)\big)$ (see Table \ref{pt}). Take $P_{X_1}=\textbf{MZTG}\big(\frac{\vT}{1+\vT}\big)$.\smallskip

\noindent \textbf{(a)}  Let $\big(\rho_1(\theta),\rho_2(\theta)\big)=\big(\theta,e^{-\theta}\big)$ for any $\theta\in D$. Since  $\al(\vT)=1-e^{-\vT}$ and $b(\vT)=(\vT-1)\cdot\big(1-e^{-\vT}\big)$, we may apply Theorem \ref{pan2} to get 
	\[
	g(x)=\begin{cases}
		\E_P\big[e^{-\vT^2}\big]	 & \text{, if } x=0\\
		\sum_{y=1}^xD_{x,y}\cdot g(x-y)& \text{, if } x\in \N
	\end{cases}, 
	\]
	where $D_{x,y}= \E_{\rho_{x-y}}\Big[\big(1+(\vT-1)\cdot\frac{y}{x}\big)\cdot \frac{(1-e^{-\vT})\cdot\vT}{(1+\vT)^{y}} \Big]$ for any  $x\in \N$ and $y\in\{1,\ldots, x\}$ (see Figure \ref{fig47a}).
	\begin{figure}[H]
		\centering
		\begin{subfigure}[t]{0.32\textwidth}
			\subcaption*{\textbf{IG}$(2,5)$}
			\includegraphics[width=0.95\linewidth]{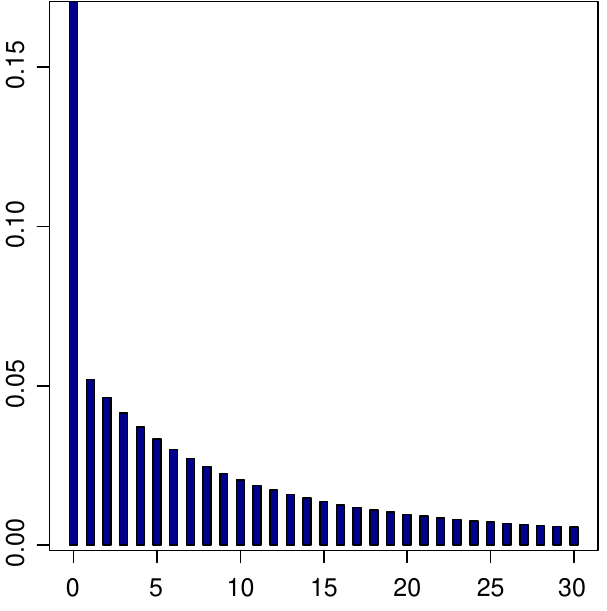}
			\label{fig47a1}
		\end{subfigure}
		~ 
		\begin{subfigure}[t]{0.32\textwidth}
			\subcaption*{\textbf{Ga}$(5,4)$}
			\includegraphics[width=0.95\linewidth]{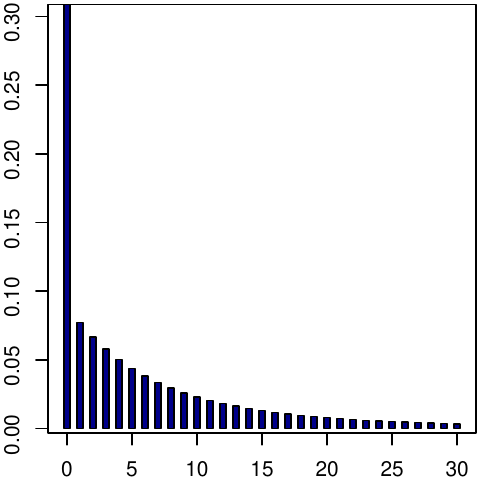}
			\label{fig47a2}
		\end{subfigure}
		~ 
		\begin{subfigure}[t]{0.32\textwidth}
			\subcaption*{\textbf{Be}$(5,5)$}
			\includegraphics[width=0.95\linewidth]{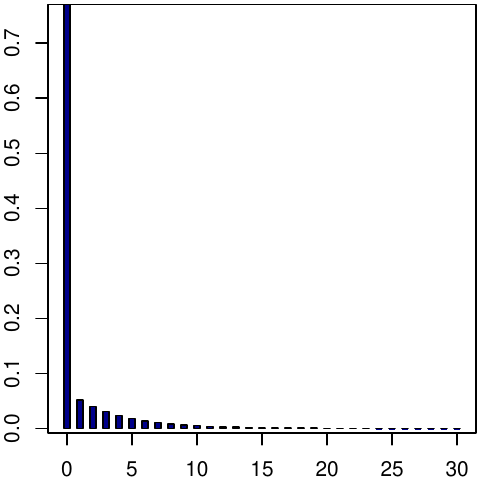}
			\label{fig47a3}
		\end{subfigure}
		\caption{The pmf of $ S $ for $P_N=\textbf{MNB}\big(\vT,e^{-\vT} \big)$  and $P_{X_1}=\textbf{MZTG}\big(\frac{\vT}{1+\vT}\big)$.}
		\label{fig47a}
	\end{figure}
	
	\noindent \textbf{(b)} Now take $\big(\rho_1(\theta),\rho_2(\theta)\big)=\big(2 \theta,\frac{1}{1+\theta}\big)$ for any $\theta\in D$. As $\al(\vT)=\frac{\vT}{1+\vT}$ and $b(\vT)=(2\vT-1)\cdot\frac{\vT}{1+\vT}$, we can apply Theorem \ref{pan2} to obtain  
	\[
	g(x)=\begin{cases}
		\E_P\Big[\big(\frac{1}{1+\vT}\big)^{2\vT}\Big] & \text{, if } x=0\\
		\sum_{y=1}^xD_{x,y}\cdot g(x-y)& \text{, if } x\in \N
	\end{cases}, 
	\]
	where $D_{x,y}= \E_{\rho_{x-y}}\Big[\big(1+(2\vT-1)\cdot\frac{y}{x}\big)\cdot \frac{\vT^2}{(1+\vT)^{y+1}}\Big]$ for any  $x\in \N$ and $y\in\{1,\ldots, x\}$ (see Figure \ref{fig47b}).		
	\begin{figure}[H]
		\centering
		\begin{subfigure}[t]{0.32\textwidth}
			\subcaption*{\textbf{IG}$(2,5)$}
			\includegraphics[width=0.95\linewidth]{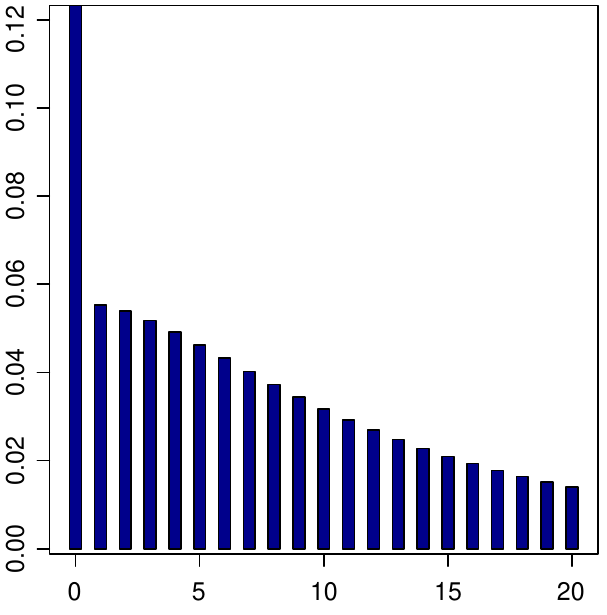}
			\label{fig47b1}
		\end{subfigure}
		~ 
		\begin{subfigure}[t]{0.32\textwidth}
			\subcaption*{\textbf{Ga}$(5,4)$}
			\includegraphics[width=0.95\linewidth]{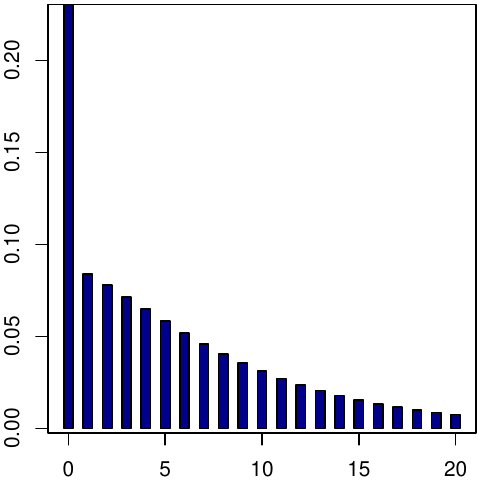}
			\label{fig47b2}
		\end{subfigure}
		~ 
		\begin{subfigure}[t]{0.32\textwidth}
			\subcaption*{\textbf{Be}$(5,5)$}
			\includegraphics[width=0.95\linewidth]{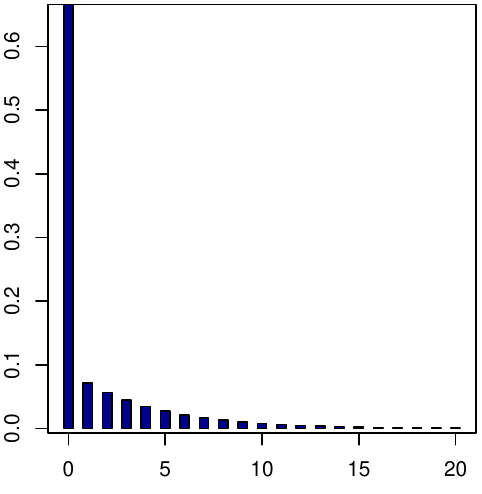}
			\label{fig47b3}
		\end{subfigure}
		\caption{The pmf of $ S $ with $P_N=\textbf{MNB}\big(2\vT,\frac{1}{1+\vT} \big)$ and $P_{X_1}=\textbf{MZTG}\big(\frac{\vT}{1+\vT}\big)$.}
		\label{fig47b}
	\end{figure}
	
	\noindent \textbf{(c)}  Let $D=(0,1)$ and take  $\big(\rho_1(\theta),\rho_2(\theta)\big)=(r,\theta)$ for any $\theta\in D$, where $r>0$. Since $\al(\vT)=1-\vT$ and $b(\vT)=(r-1)\cdot(1-\vT)$, we can apply Theorem \ref{pan2} to obtain  
	\[
	g(x)=\begin{cases}
		\E_P\big[\vT^{r}\big] & \text{, if } x=0\\
		\sum_{y=1}^xD_{x,y}\cdot g(x-y)& \text{, if } x\in \N
	\end{cases}, 
	\]
	where $D_{x,y}=\big(1+(r-1)\cdot\frac{y}{x}\big)\cdot \E_{\rho_{x-y}}\Big[\frac{(1-\vT)\cdot\vT}{(1+\vT)^{y}}\Big]$ for any  $x\in \N$ and $y\in\{1,\ldots, x\}$. Figure \ref{fig47c} illustrates $g$ for  $P_\vT=\textbf{Be}(\al,\be)$, where $\al,\be>0$, and different values of $r$.
	\begin{figure}[H]
		\centering
		\begin{subfigure}[t]{0.32\textwidth}
			\subcaption*{$r=2$; $\al=2$; $\be=4$}
			\includegraphics[width=0.95\linewidth]{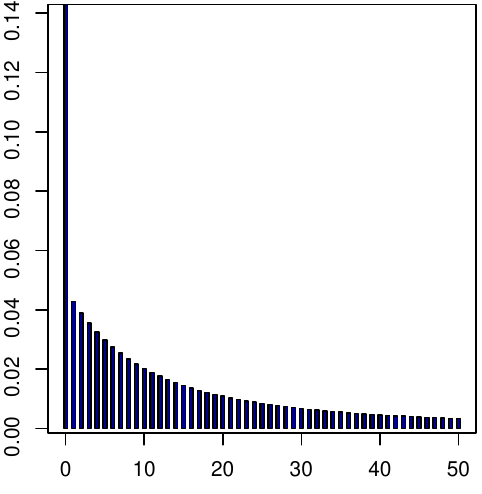}
			\label{fig47c1}
		\end{subfigure}
		~ 
		\begin{subfigure}[t]{0.32\textwidth}
			\subcaption*{$r=5$; $\al=4$; $\be=4$}
			\includegraphics[width=0.95\linewidth]{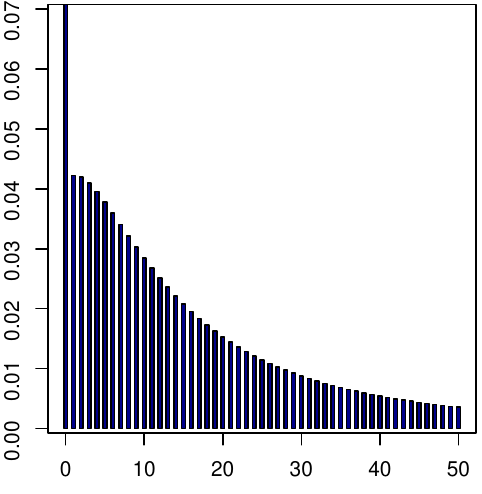}
			\label{fig47c2}
		\end{subfigure}
		~ 
		\begin{subfigure}[t]{0.32\textwidth}
			\subcaption*{$r=10$; $\al=7$; $\be=5$}
			\includegraphics[width=0.95\linewidth]{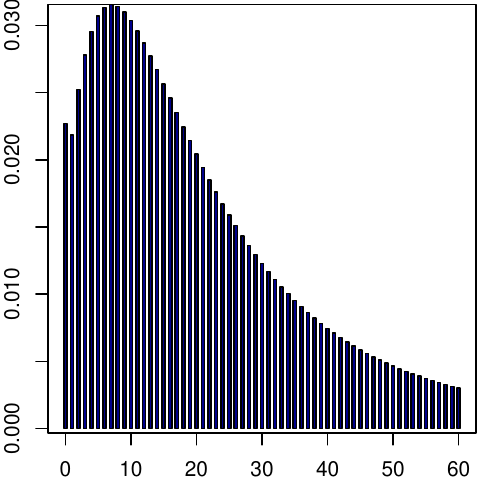}
			\label{fig47cd3}
		\end{subfigure}
		\caption{The pmf of $ S $ for different values of $r$, $\alpha$ and $\beta$.}
		\label{fig47c}
	\end{figure}
	
	\noindent \textbf{(d)}  Let $D=(0,\infty)$ and take  $\big(\rho_1(\theta),\rho_2(\theta)\big)=(\theta,p)$ for any $\theta\in D$, where $p\in(0,1)$. As $\al(\vT)=1-p$ and $b(\vT)=(\vT-1)\cdot(1-p)$, we can apply Theorem \ref{pan2} to obtain  
	\[
	g(x)=\begin{cases}
		\E_P\big[p^\vT\big] & \text{, if } x=0\\
		\sum_{y=1}^xD_{x,y}\cdot g(x-y)& \text{, if } x\in \N
	\end{cases}, 
	\]
	where $D_{x,y}= (1-p)\cdot\E_{\rho_{x-y}}\Big[\big(1+(\vT-1)\cdot\frac{y}{x}\big)\cdot \frac{\vT}{(1+\vT)^{y}}\Big]$ for any  $x\in \N$ and $y\in\{1,\ldots, x\}$. Figure \ref{fig47d} depicts $g$ for $P_\vT=\textbf{Ga}(\al,\be)$, where $\al,\be>0$, and various values of $p$.
	\begin{figure}[H]
		\centering
		\begin{subfigure}[t]{0.32\textwidth}
			\subcaption*{$p=1/3$; $\al=2$; $\be=1$}
			\includegraphics[width=0.95\linewidth]{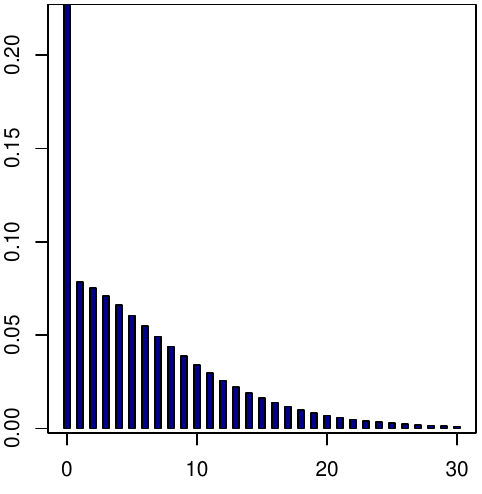}
			\label{fig47d1}
		\end{subfigure}
		~ 
		\begin{subfigure}[t]{0.32\textwidth}
			\subcaption*{$p=1/2$; $\al=5$; $\be=1/4$}
			\includegraphics[width=0.95\linewidth]{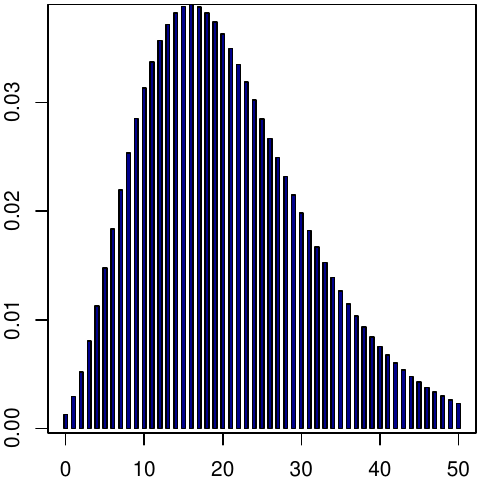}
			\label{fig47d2}
		\end{subfigure}
		~ 
		\begin{subfigure}[t]{0.32\textwidth}
			\subcaption*{$p=2/3$; $\al=3$; $\be=1/3$}
			\includegraphics[width=0.95\linewidth]{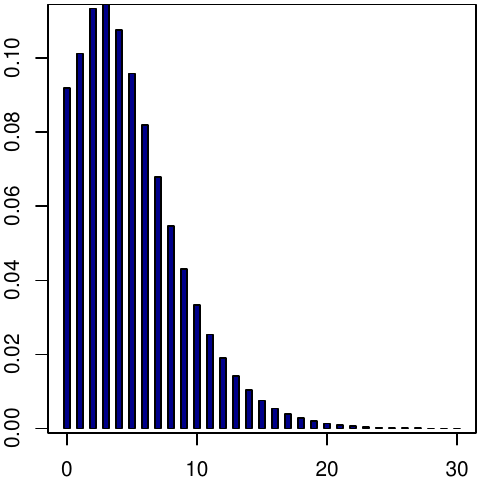}
			\label{fig47d3}
		\end{subfigure}
		\caption{The pmf of $ S $ for different values of $p$, $\alpha$ and $\beta$.}
		\label{fig47d}
	\end{figure}
	\noindent\textbf{(e)} Assume that $N$ and $\vT$ are $P$-independent and let $\big(\rho_1(\theta),\rho_2(\theta)\big)=(r,p)$ for any $\theta\in D$, where $r>0$ and $p\in(0,1)$. Since  $\al=1-p$ and $b=(r-1)\cdot(1-p)$, we can apply Corollary \ref{exch} to obtain  
	\[
	g(x)=\begin{cases}
		p^r & \text{, if } x=0\\
		\sum_{y=1}^xD_{x,y}(F)\cdot g(x-y)& \text{, if } x\in \N
	\end{cases}, 
	\]
	where $D_{x,y}=(1-p)\cdot\big(1+(r-1)\cdot\frac{y}{x}\big)\cdot\E_{\rho_{x-y}}\Big[\frac{\vT}{(1+\vT)^{y}}\Big]$ for any  $x\in \N$ and $y\in\{1,\ldots, x\}$ (see Figure \ref{fig47e}).
	\begin{figure}[H]
		\centering
		\begin{subfigure}[t]{0.32\textwidth}
			\subcaption*{\textbf{IG}$(2,5)$}
			\includegraphics[width=0.95\linewidth]{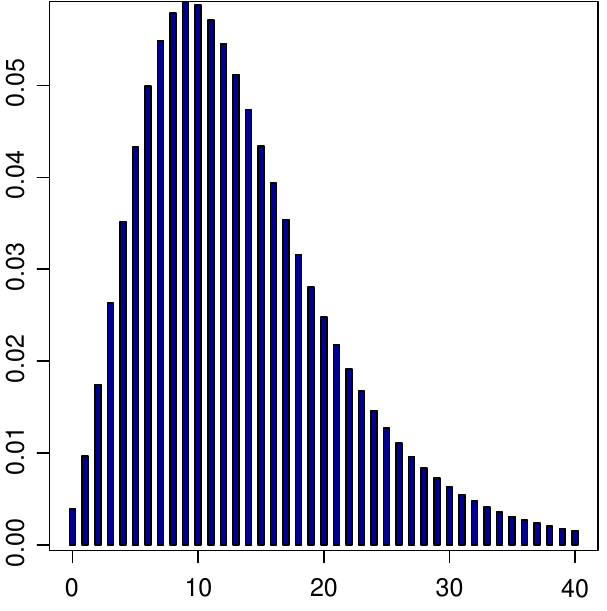}
			\label{fig47e1}
		\end{subfigure}
		~ 
		\begin{subfigure}[t]{0.32\textwidth}
			\subcaption*{\textbf{Ga}$(5,4)$}
			\includegraphics[width=0.95\linewidth]{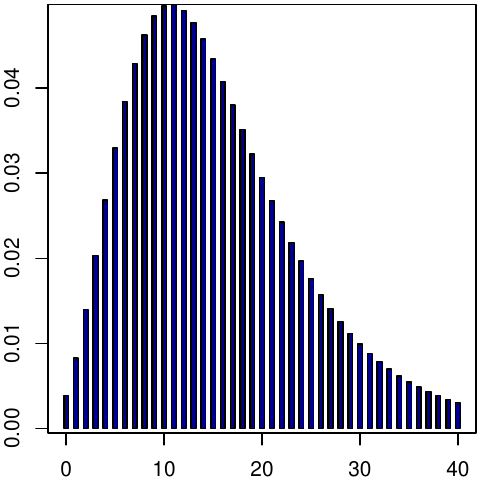}
			\label{fig47e2}
		\end{subfigure}
		~ 
		\begin{subfigure}[t]{0.32\textwidth}
			\subcaption*{\textbf{Be}$(5,5)$}
			\includegraphics[width=0.95\linewidth]{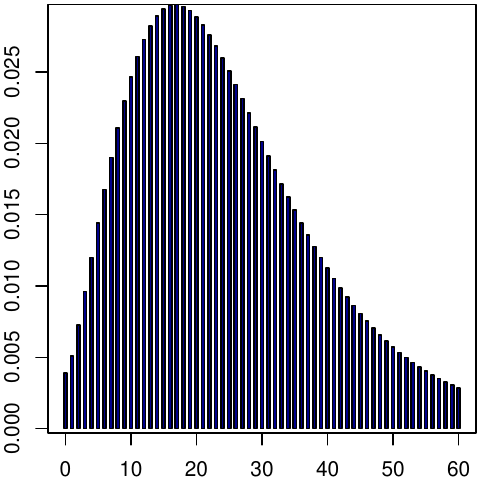}
			\label{fig47e3}
		\end{subfigure}
		\caption{The pmf of $ S $ with $P_N=\textbf{NB}(8,1/2)$ and $P_{X_1}=\textbf{MZTG}\big(\frac{\vT}{1+\vT}\big)$.}
		\label{fig47e}
	\end{figure}
\end{ex}  

\begin{ex}\label{CMB}
	\normalfont
	Let $D\subseteq (0,\infty)$, $P_N=\textbf{MB}\big(n,z_2(\vT)\big)$, where $n\in\N$  and $z_2$ is a  $\B(D)$-$\B(0,1)$-measurable function. As $P_N\in {\bf Panjer}(a(\vT),b(\vT);0)$, we get $\al(\vT)=-\frac{z_2(\vT)}{1-z_2(\vT)}$ and $b(\vT)=(n+1)\cdot \frac{z_2(\vT)}{1-z_2(\vT)}$ (see Table \ref{pt}). Take $P_{X_1}=\textbf{MZTG}\big(\frac{\vT}{1+\vT}\big)$.\smallskip 
	
	\noindent \textbf{(a)}  Let $D=(0,1)$ and take $z_2(\theta)=1-\theta$ for any $\theta\in D$. Since  $\al(\vT)=-\frac{1-\vT}{\vT}$ and $b(\vT)=(n+1)\cdot\frac{1-\vT}{\vT}$, we may apply Theorem \ref{pan2} to get 
	\[
	g(x)=\begin{cases}
		\E_P\big[\vT^n\big]	 & \text{, if } x=0\\
		\sum_{y=1}^xD_{x,y}\cdot g(x-y)& \text{, if } x\in \N
	\end{cases}, 
	\] 
where $D_{x,y}=\big(-1+(n+1)\cdot\frac{y}{x}\big)\cdot\E_{\rho_{x-y}}\big[\frac{1-\vT}{(1+\vT)^{y}}\big]$ for any  $x\in \N$ and $y\in\{1,\ldots, x\}$. Figure \ref{fig48a} depicts $g$ for $P_\vT=\textbf{Be}(\al,\be)$, where $\al,\be>0$, and different values of $n$.
	\begin{figure}[H]
		\centering
		\begin{subfigure}[t]{0.32\textwidth}
			\subcaption*{$n=2$; $\al=3$; $\be=5$}
			\includegraphics[width=.95\linewidth]{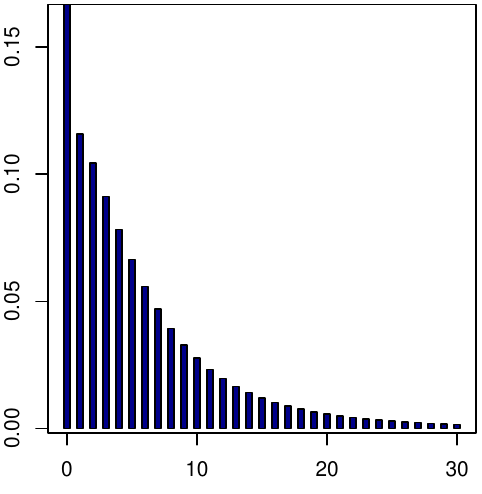}
			\label{fig48a1}
		\end{subfigure}
		~ 
		\begin{subfigure}[t]{0.32\textwidth}
			\subcaption*{$n=5$; $\al=2$; $\be=3$}
			\includegraphics[width=0.95\linewidth]{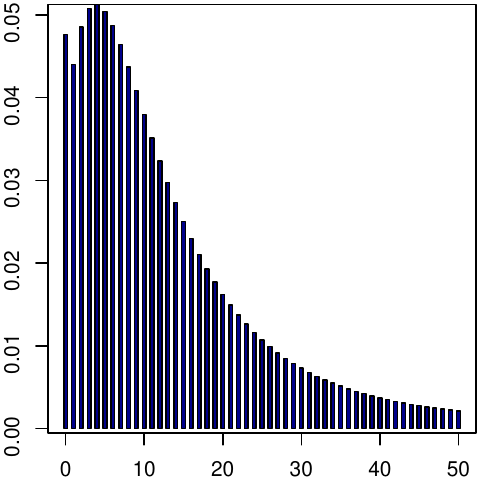}
			\label{fig48a2}
		\end{subfigure}
		~ 
		\begin{subfigure}[t]{0.32\textwidth}
			\subcaption*{$n=10$; $\al=6$; $\be=8$}
			\includegraphics[width=0.95\linewidth]{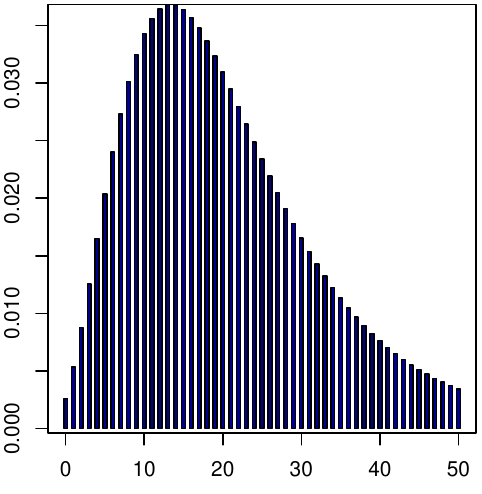}
			\label{fig48a3}
		\end{subfigure}
		\caption{The pmf of $ S $ for different values of $n$, $\alpha$ and $\beta$.}
		\label{fig48a}
	\end{figure}
	\noindent \textbf{(b)}  Let $D=(0,\infty)$ and take $z_2(\theta)=e^{-\theta}$  for any $\theta\in D$. As $\al(\vT)=-\frac{e^{-\vT}}{1-e^{-\vT}}$ and $b(\vT)=(n+1)\cdot\frac{e^{-\vT}}{1-e^{-\vT}}$, we may apply Theorem \ref{pan2}, together with Remark \ref{com}(a), to get 
	\[
	g(x)=\begin{cases}
		\E_P\Big[\big(1-e^{-\vT}\big)^n\Big]	 & \text{, if } x=0\\
		\sum_{y=1}^xD_{x,y}\cdot g(x-y)& \text{, if } x\in \N
	\end{cases}, 
	\]
	where $D_{x,y}=\big(-1+(n+1)\cdot\frac{y}{x}\big)\cdot\E_{\rho_{x-y}}\Big[\frac{\vT\cdot e^{-\vT}}{(1-e^{-\vT})\cdot(1+\vT)^{y}} \Big]$ for any  $x\in \N$ and $y\in\{1,\ldots, x\}$. Figure \ref{fig48b} illustrates $g$ for $P_\vT=\textbf{Ga}(\al,\be)$, where $\al,\be>0$, and various values of $n$.
	\begin{figure}[H]
		\centering
		\begin{subfigure}[t]{0.32\textwidth}
			\subcaption*{$n=2$; $\al=3/4$; $\be=3$}
			\includegraphics[width=0.95\linewidth]{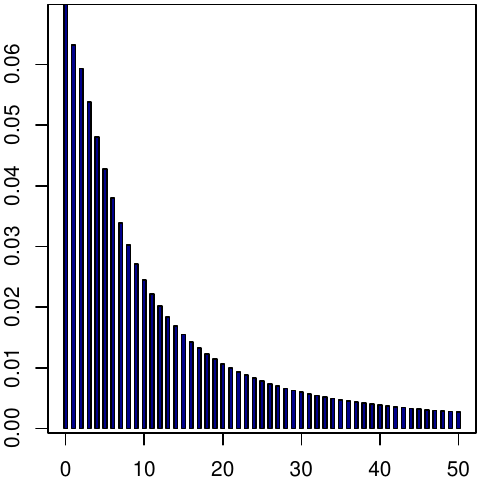}
			\label{fig48b1}
		\end{subfigure}
		~ 
		\begin{subfigure}[t]{0.32\textwidth}
			\subcaption*{$n=5$; $\al=3$; $\be=4$}
			\includegraphics[width=0.95\linewidth]{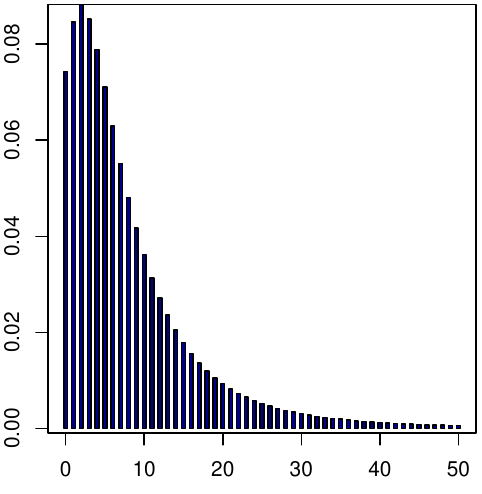}
			\label{fig48b2}
		\end{subfigure}
		~ 
		\begin{subfigure}[t]{0.32\textwidth}
			\subcaption*{$n=10$; $\al=4$; $\be=5$}
			\includegraphics[width=0.95\linewidth]{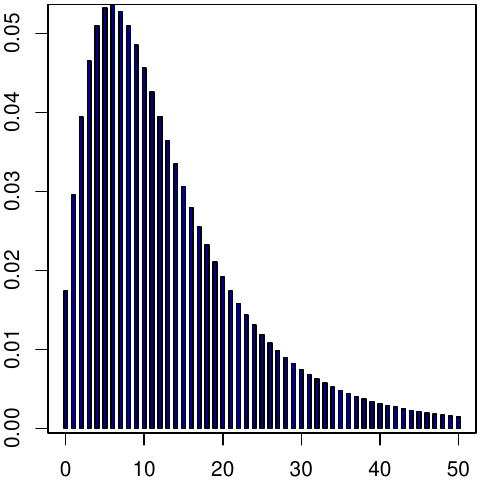}
			\label{fig48b3}
		\end{subfigure}
		\caption{The pmf of $ S $ for different values of $n$, $\alpha$ and $\beta$.}
		\label{fig48b}
	\end{figure}
	\noindent \textbf{(c)}  Let $D=(0,\infty)$ and take $z_2(\theta)=\frac{1}{1+\theta}$  for any $\theta\in D$. Since $\al(\vT)=-\frac{1}{\vT}$ and $b(\vT)=\frac{(n+1)}{\vT}$, we may apply Theorem \ref{pan2}  to get 
	\[
	g(x)=\begin{cases}
		\E_P\Big[\big(\frac{\vT}{1+\vT}\big)^n\Big]	 & \text{, if } x=0\\
		\sum_{y=1}^xD_{x,y}\cdot g(x-y)& \text{, if } x\in \N
	\end{cases}, 
	\]
	where $D_{x,y}=\big(-1+(n+1)\cdot\frac{y}{x}\big)\cdot\E_{\rho_{x-y}}\Big[\frac{1}{(1+\vT)^{y}}\Big]$ for any $x\in \N$ and $y\in\{1,\ldots, x\}$. Figure \ref{fig48c} depicts $g$ for $P_\vT=\textbf{IG}(\mu,\varphi)$, where $\mu,\varphi>0$, and different values of $n$.
	\begin{figure}[H]
		\centering
		\begin{subfigure}[t]{0.32\textwidth}
			\subcaption*{$n=2$; $\mu=2$; $\varphi=1/2$}
			\includegraphics[width=0.95\linewidth]{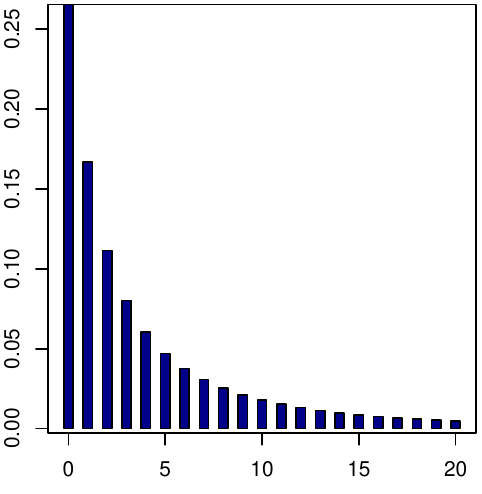}
			\label{fig48c1}
		\end{subfigure}
		~ 
		\begin{subfigure}[t]{0.32\textwidth}
			\subcaption*{$n=5$; $\mu=1$; $\varphi=3$}
			\includegraphics[width=0.95\linewidth]{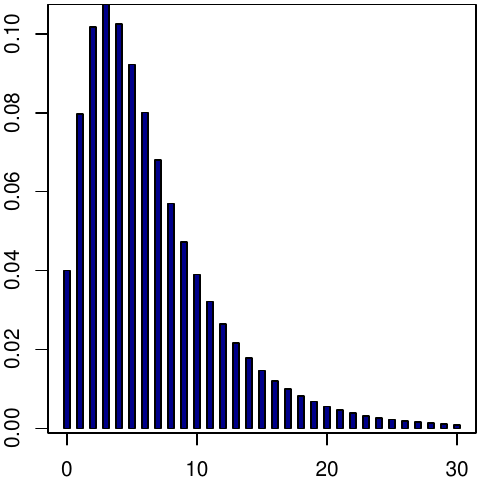}
			\label{fig48c2}
		\end{subfigure}
		~ 
		\begin{subfigure}[t]{0.32\textwidth}
			\subcaption*{$n=10$; $\mu=0.5$; $\varphi=10$}
			\includegraphics[width=0.95\linewidth]{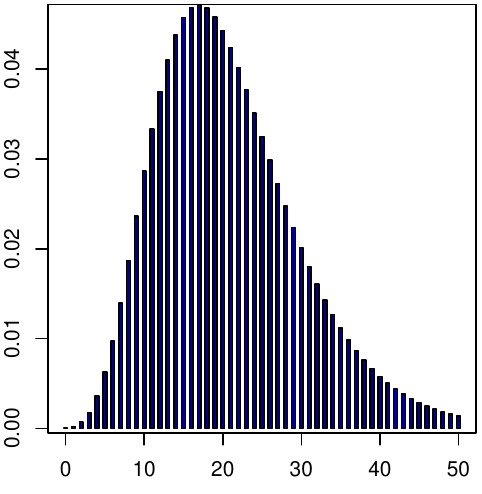}
			\label{fig48c3}
		\end{subfigure}
		\caption{The pmf of $ S $ for different values of $n$, $\mu$ and $\varphi$.}
		\label{fig48c}
	\end{figure}
\end{ex}
\section{Concluding Remarks}\label{conc}

When dealing with inhomogeneous insurance portfolios, mixtures of claim number distributions are frequently used to model the claim counts. In actuarial practice, the usual independence assumptions seem to be unrealistic. In particular, the mutual independence between claims sizes and counts appears to be implausible, especially when considering inhomogeneous portfolios. The above led us to the development of this work.

This paper aims to investigate the mixed counterpart of the original ${\bf Panjer}(a,b;0)$ class and the corresponding compound distributions by relaxing the usual independence assumptions in two ways. More specifically, we allow the parameters of the claim number and the claim size distributions to be random, and we assume that the claim size process is conditionally i.i.d. and conditionally mutually independent of the claim counts. 

Based on the above assumptions and the fundamental concept of regular conditional probabilities, this work provides a characterization for the members of ${\bf Panjer}(a(\vT),b(\vT);0)$ class, as well as two recursive algorithms for $ P_N $ and $ P_S $, respectively, where the recursion for $ P_S $ also covers the case of a compound Panjer distribution with exchangeable claims. Our results are accompanied by various numerical examples to highlight the practical relevance of this work.

It is worth mentioning that the members of the ${\bf Panjer}(a(\vT),b(\vT);0)$ class may be of particular interest in non-life insurance problems, for example, when insurance data exhibit overdispersion. 
Finally, this research allows us to consider some possible extensions of actuarial interest beyond the scopes of this paper, such as the study of the mixed counterpart of ${\bf Panjer}(a,b;k)$ for $k\geq 1$ and the derivation of De Pril's recursion for the moments of $S$ (see \cite{dp}, Theorem, p. 118).


\begin{thebibliography}{99} 
	
	\bibitem[\protect\astroncite{Albrecher \& Boxma}{2004}]{albo} \textbf{Albrecher, H. \& Boxma, O.~J.} (2004). A ruin model with dependence between claim sizes and claim intervals. \textit{Insurance: Mathematics and Economics}, \textbf{35}(2), 245--254.
	
	\bibitem[\protect\astroncite{Albrecher \& Teugels}{2006}]{alte} \textbf{Albrecher, H. \& Teugels, J.~L.} (2006). Exponential behavior in the presence of dependence in risk theory. \textit{Journal of Applied Probability}, \textbf{43}(1), 257--273.
	
	\bibitem[\protect\astroncite{Albrecher et al.}{2011}]{acl} \textbf{Albrecher, H., Constantinescu, C. \&  Loisel, S.} (2011). Explicit ruin formulas for models with dependence among risks. \textit{Insurance: Mathematics and Economics}, \textbf{48}(2), 265--270.
	
	\bibitem[\protect\astroncite{Antzoulakos \& Chadjiconstantinidis}{2004}]{anch} \textbf{Antzoulakos, D.~L. \& Chadjiconstantinidis, S.} (2004). On mixed and compound mixed Poisson distributions. \textit{Scandinavian Actuarial Journal}, \textbf{2004}(3), 161--188. 
	
	\bibitem[\protect\astroncite{Beal}{1940}]{beal} \textbf{Beal, G.}  (1940). The fit and significance of contagious distributions when applied to observations on larval insects. \textit{Ecology}, \textbf{21}(4), 460--474.
	
	\bibitem[\protect\astroncite{Chang \& Pollard}{1997}]{chpo} \textbf{Chang, J.~T. \& Pollard, D.} (1997). Conditioning as disintegration. \textit{Statistica Neerlandica}, \textbf{51}(3), 287--317.
	
	\bibitem[\protect\astroncite{Cohn}{2013}]{Co} \textbf{Cohn, D.~L.} (2013).  \textit{Measure Theory}, 2nd ed. Birkh\"{a}user Advanced Texts. 
	
	\bibitem[\protect\astroncite{De Pril}{1986}]{dp} \textbf{De Pril, N.} (1986). Moments of a class of compound distributions . \textit{Scandinavian Actuarial Journal}, \textbf{1986}(2), 117--120.
	
	\bibitem[\protect\astroncite{Fackler}{2023}]{fac} \textbf{Fackler M.} (2023). Panjer class revisited: one formula for the distributions of the Panjer $(a,b,n)$ class. \textit{Annals of Actuarial Science}, \textbf{17}(1), 145--169. 
	
	\bibitem[\protect\astroncite{Faden}{1985}]{faden}	\textbf{Faden, A.~M.} (1985). The existence of regular conditional probabilities: Necessary and sufficient conditions. \textit{Annals of Probability}, \textbf{13}(1), 288--298.
	
	
	\bibitem[\protect\astroncite{Fremlin}{2013}]{fr4} \textbf{Fremlin, D.~H.} (2013). \textit{Measure theory} (Vol. 4). Torres Fremlin (Ed.).
	
	\bibitem[\protect\astroncite{Gençtürk \& Yiğiter}{2016}]{gy} \textbf{Gençtürk, Y. \& Yiğiter, A.} (2016). Modelling claim number using a new mixture model: negative binomial gamma distribution. \textit{Journal of Statistical Computation and Simulation}, \textbf{86}(10), 1829--1839.
	
	\bibitem[\protect\astroncite{Gerhold et al.}{2010}]{gsw} \textbf{Gerhold, S., Schmock, U. \&  Warnung, R.} (2010). A generalization of Panjer’s recursion and numerically stable risk aggregation. \textit{Finance and Stochastics}, \textbf{14}(1), 81--128.
	
	\bibitem[\protect\astroncite{Ghitany et al.}{2008}]{gan} \textbf{Ghitany, M.~E., Atieh, B. \& Nadarajah, S.} (2008). Lindley distribution and its application. \textit{Mathematics and computers in simulation}, \textbf{78}(4), 493--506.
	
	\bibitem[\protect\astroncite{Gómez-Déniz et al.}{2004}]{gdsco} \textbf{Gómez-Déniz, E., Sarabia, J.~M. \& Calderín-Ojeda, E.} (2008). Univariate and multivariate versions of the negative binomial-inverse Gaussian distributions with applications. \textit{Insurance: Mathematics and Economics}, \textbf{42}(1), 39--49.
	
	\bibitem[\protect\astroncite{Grandell}{1997}]{gr} \textbf{Grandell, J.} (1997). \textit{Mixed Poisson Processes}, Chapman \& Hall.
	
	\bibitem[\protect\astroncite{Johnson et al.}{2005}]{jkk} \textbf{Johnson, N.~L., Kotz, S. \& Kemp, A.~W.} (2005). \textit{Univariate Discrete Distributions}, 3rd ed. John Wiley and Sons, New York.
	
	\bibitem[\protect\astroncite{Hess et al.}{2002}]{hls}  \textbf{Hess, K.~T.,   Liewald, A. \& Schmidt, K.~D.} (2002) An extension of Panjer's recursion. \textit{ASTIN Bulletin: The Journal of the IAA}, \textbf{32}(2), 61--80. 
	
	\bibitem[\protect\astroncite{Hesselager}{1994}]{hessa}   \textbf{Hesselager, O.} (1994).  A recursive procedure for calculation of some compound distributions. \textit{ASTIN Bulletin: The Journal of the IAA}, \textbf{24}(1), 19--32.
	
	\bibitem[\protect\astroncite{Hesselager}{1996}]{hessb}   \textbf{Hesselager, O.} (1996). A recursive procedure for calculation of some mixed compound Poisson distributions. \textit{Scandinavian Actuarial Journal} \textbf{1996}(1), 54--63.
	
	\bibitem[\protect\astroncite{Karlis \& Xekalaki}{2005}]{kaxe} \textbf{Karlis, D. \& Xekalaki, E.} (2005). Mixed poisson distributions. \textit{International Statistical Review/Revue Internationale de Statistique}, \textbf{73}(1), 35--58.
	
	\bibitem[\protect\astroncite{Kolev \& Paiva}{2008}]{kopa} \textbf{Kolev, N. \& Paiva, D.} (2008). Random sums of exchangeable variables and actuarial applications. \textit{Insurance: Mathematics and Economics}, \textbf{42}(1), 147--153.
	
	\bibitem[\protect\astroncite{Lyberopoulos \&  Macheras}{2012}]{lm1v} \textbf{Lyberopoulos, D.~P. \& Macheras, N.~D.}  (2012). Some characterizations of mixed Poisson processes. \textit{Sankhy\={a} A}, {\bf 74}(1), 57--79. 
	
	\bibitem[\protect\astroncite{Lyberopoulos \& Macheras}{2013}]{lm5}  \textbf{Lyberopoulos, D.~P. \& Macheras, N.~D.} (2013). A construction of mixed Poisson processes via disintegrations. \textit{Mathematica Slovaca}, \textbf{63}(1), 167--182.
	
	\bibitem[\protect\astroncite{Lyberopoulos \& Macheras}{2019}]{lm3ar}\textbf{Lyberopoulos, D.~P. \& Macheras, N.~D.} (2019). A characterization of martingale-equivalent compound mixed Poisson process. arXiv preprint arXiv:1905.07629.
	
	\bibitem[\protect\astroncite{Lyberopoulos \& Macheras}{2021}]{lm3} \textbf{Lyberopoulos, D.~P. \& Macheras, N.~D.}  (2021). A characterization of martingale-equivalent mixed compound  Poisson processes. \textit{Annals of Applied Probabability}, \textbf{31}(2), 778--805.
	
	\bibitem[\protect\astroncite{Lyberopoulos \& Macheras}{2022}]{lm6z3}  \textbf{Lyberopoulos, D.~P. \& Macheras, N.~D.} (2022). Some characterizations  of mixed renewal processes. \textit{Mathematica Slovaca} \textbf{72}(1), 197--216.
	
	\bibitem[\protect\astroncite{Lyberopoulos et al.}{2019}]{lmt1} \textbf{Lyberopoulos, D.~P., Macheras, N.~D. \& Tzaninis S.~M.} (2019). On the equivalence of various definitions of mixed Poisson processes. \textit{Mathematica Slovaca}, \textbf{69}(2), 453--468.
	
	
	\bibitem[\protect\astroncite{Mansoor et al.}{2020}]{matal} \textbf{Mansoor, M., Tahir, M.~H., Cordeiro, G.~M., Ali, S. \& Alzaatreh, A.} (2020). The Lindley negative-binomial distribution: Properties, estimation and applications to lifetime data. \textit{Mathematica Slovaca},  \textbf{70}(4), 917--934.
	
	\bibitem[\protect\astroncite{Nadarajah \& Kotz}{2006a}]{nakp1} \textbf{Nadarajah, S. \& Kotz, S.} (2006a). Compound mixed Poisson distributions I. \textit{Scandinavian Actuarial Journal}, \textbf{2006}(3), 141--162.
	
	\bibitem[\protect\astroncite{Nadarajah \& Kotz}{2006b}]{nakp2} \textbf{Nadarajah, S. \& Kotz, S.} (2006b) Compound mixed Poisson distributions II. \textit{Scandinavian Actuarial Journal}, \textbf{2006}(3), 163--181.
	
	\bibitem[\protect\astroncite{Panjer}{1981}]{pa} \textbf{Panjer, H.~H.} (1981). Recursive evaluation of a family of compound distributions. \textit{ASTIN Bulletin: The Journal of the IAA}, \textbf{12}(1), 22--26.
	
	\bibitem[\protect\astroncite{Schmidt}{2012}]{Sc}   \textbf{Schmidt, K.~D.}  (2012). \textit{Lectures on Risk Theory}. Springer Science \& Business Media.
	
	\bibitem[\protect\astroncite{Schmidt}{2014}]{Sc2014}   \textbf{Schmidt, K.~D.}  (2014). On inequalities for moments and the covariance of monotone functions.  \textit{Insurance: Mathematics and Economics}, \textbf{55}, 91--95. 
	
	\bibitem[\protect\astroncite{Schr\"{o}ter}{1990}]{sch} \textbf{Schr\"{o}ter, K.} (1990). On  a  class  of  counting  distributions  and  recursions  for  related  compound  distributions.   \textit{Scandinavian Actuarial Journal}, \textbf{1990}(2-3), 161--175. 
	
	\bibitem[\protect\astroncite{Stoyanov}{2013}]{sto}  \textbf{Stoyanov, J.~M.}  (2013). \textit{Counterexamples in Probability}, 3rd ed.. Dover Publications.
	
	\bibitem[\protect\astroncite{Sundt}{1992}]{su} \textbf{Sundt, B.} (1992). On some extensions of Panjer's class of counting distributions. \textit{ASTIN Bulletin: The Journal of the IAA}, \textbf{22}(1), 61--80.
	
	\bibitem[\protect\astroncite{Sundt \& Jewell}{1981}]{sj} \textbf{Sundt, B. \& Jewell, W.~S.} (1981). Further results on recursive evaluation of compound distributions. \textit{ASTIN Bulletin: The Journal of the IAA}, \textbf{12}(1), 27--39.
	
	\bibitem[\protect\astroncite{Sundt \& Vernic}{2004}]{sv2004} \textbf{Sundt, B. \& Vernic, R.} (2004). Recursions for compound mixed multivariate Poisson distributions. \textit{Blätter der DGVFM}, \textbf{26}(4), 665--691. 
	
	\bibitem[\protect\astroncite{Sundt \& Vernic}{2009}]{sv} \textbf{Sundt, B. \& Vernic, R.} (2009). \textit{Recursions for convolutions and compound distributions with insurance applications}. Springer Science \& Business Media.
	
	
	\bibitem[\protect\astroncite{Tzaninis \& Macheras}{2023}]{mt2} \textbf{Tzaninis, S.~M. \& Macheras, N.~D.} (2023). A characterization of progressively equivalent probability measures preserving the structure of a compound mixed renewal process. \textit{ALEA, Latin American Journal of Probability and Mathematical Statistics}, \textbf{20}, 225--247.
	
	\bibitem[\protect\astroncite{Tzougas et al.}{2019}]{tzougas} \textbf{Tzougas, G., Hoon, W.~L.\& Lim, J.~M.} (2019). The negative binomial-inverse Gaussian regression model with an application to insurance ratemaking. \textit{European Actuarial Journal}, \textbf{9}, 323--344.
	
	\bibitem[\protect\astroncite{Tzougas et al.}{2021}]{tzougas2} \textbf{Tzougas, G., Hong, N. \& Ho, R.} (2021). Mixed poisson regression models with varying dispersion arising from non-conjugate mixing distributions. \textit{Algorithms}, \textbf{15}(1), 16.
	
	\bibitem[\protect\astroncite{Wang}{2011}]{zwang} \textbf{Wang, Z.} (2011). One mixed negative binomial distribution with application. \textit{Journal of Statistical Planning and Inference}, \textbf{141}(3), 1153--1160.
	
	\bibitem[\protect\astroncite{Wang \& Sobrero}{1994}]{ws}\textbf{Wang, S. \& Sobrero, M.} (1994).  Further results on Hesselager’s recursive procedure for calculation of some compound distributions. \textit{ASTIN Bulletin: The Journal of the IAA}, \textbf{16}(2), 161--166.
	
	\bibitem[\protect\astroncite{Willmot}{1986}]{wil1986} \textbf{Willmot, G.~E.} (1986). Mixed compound Poisson distributions. \textit{ASTIN Bulletin: The Journal of the IAA}, \textbf{16}(1), 59--79.  
	
	\bibitem[\protect\astroncite{Willmot}{1987}]{wil1987} \textbf{Willmot, G.~E.} (1987). The Poisson-Inverse Gaussian distribution as an alternative to the negative binomial. \textit{Scandinavian Actuarial Journal}, \textbf{1987}(3-4), 113--127.
	
	\bibitem[\protect\astroncite{Willmot}{1988}]{wil88} \textbf{Willmot, G.~E.} (1988). Sundt and Jewell’s family of discrete distributions. \textit{ASTIN Bulletin: The Journal of the IAA}, \textbf{18}(1), 17--29.
	
	\bibitem[\protect\astroncite{Willmot}{1993}]{wil} \textbf{Willmot, G.~E.} (1993). On recursive  evaluation  of mixed Poisson  probabilities and related  quantities. \textit{Scandinavian Actuarial Journal}, \textbf{1993}(2), 114--133. 
\end{thebibliography}
\end{document}